\documentclass[11pt]{amsart}

\usepackage[english]{babel}
\usepackage[latin1]{inputenc}
\usepackage{amsthm,amsmath,amsfonts, amssymb,bbm}
\usepackage{hyperref,enumitem,stmaryrd}
\usepackage{graphicx,eepic,pstricks}

\def \RR{\mathbb R}
\def \EE{\mathbb E}
\def \PP{\mathbb P}

\def \SS{\mathbb S}

\def \sclr#1#2{\langle #1,#2\rangle}
\def \open#1{#1^{o}}
\def \close#1{\overline{#1}}

\renewcommand{\geq}{\geqslant}
\renewcommand{\leq}{\leqslant}

\usepackage{color}
\definecolor{darkgreen}{rgb}{0,0.4,0}
\definecolor{MyDarkBlue}{rgb}{0,0.08,0.50}
\definecolor{BrickRed}{rgb}{0.65,0.08,0}

\hypersetup{
colorlinks=true,       
    linkcolor=blue,          
    citecolor=red,        
    filecolor=BrickRed,      
    urlcolor=darkgreen        
}

\theoremstyle{plain}
\newtheorem{theorem}{Theorem}
\newtheorem*{theorem-sum}{Theorem}
\newtheorem{proposition}{Proposition}

\newtheorem{lemma}[proposition]{Lemma}

\theoremstyle{definition}
\newtheorem{remark}[proposition]{Remark}
\newtheorem{example}{Example}

\begin{document}

\author[R.~Garbit]{Rodolphe Garbit}
\address{Universit\'e d'Angers\\D\'epartement de Math\'ematiques\\ LAREMA\\ UMR CNRS 6093\\ 2 Boulevard Lavoisier\\49045 Angers Cedex 1\\ France}
\email{rodolphe.garbit@univ-angers.fr}
\author[K.~Raschel]{Kilian Raschel}
\address{CNRS\\ F\'ed\'eration Denis Poisson\\Universit\'e de Tours\\LMPT\\UMR CNRS 7350\\Parc de Grandmont\\ 37200 Tours\\ France}
\email{kilian.raschel@lmpt.univ-tours.fr}

\title[On the exit time from a cone for brownian motion with drift]{On the exit time from a cone for brownian motion with drift}

\thanks{}

\date{\today}

\begin{abstract} We investigate the tail distribution of the first exit time of Brownian motion with drift from a cone and find its exact asymptotics for a large class of cones. Our results show in particular that its exponential decreasing rate is a function of the distance between the drift and the cone, whereas the polynomial part in the asymptotics depends on the position of the drift with respect to the cone and its polar cone, and reflects the local geometry of the cone at the points that minimize the distance to the drift.
\end{abstract}

\maketitle

\section{Introduction}
\label{sec:Introduction}

Let $B_t$ be a $d$-dimensional Brownian motion with drift $a\in \RR^d$. For any cone $C\subset\RR^d$, define the first exit time
\begin{equation*}
\label{eq:def_exit_time}
     \tau_C =\inf\{t>0 : B_t \notin C\}.
\end{equation*}
In this article we study the probability for the Brownian motion started at $x$ not to exit $C$ before time $t$, namely,
\begin{equation}
\label{eq:def_exit_probability}
     \PP_x[\tau_C>t],
\end{equation}
and its asymptotics
\begin{equation}
\label{eq:exit_asymptotic}
     \kappa h(x) t^{-\alpha} e^{-\gamma t}(1+o(1)),\quad t\to\infty.
\end{equation}

In the literature, these problems have first been considered for Brownian motion with no drift ($a=0$). In \cite{Sp58}, Spitzer considered the case $d=2$ and obtained an explicit expression for the probability \eqref{eq:def_exit_probability} for any two-dimensional cone. He also introduced the winding number process $\theta_t=\arg B_t$ (in dimension $d=2$, the Brownian motion does not exit a given cone before time $t$ if and only if $\theta_t$ stays in some interval). He proved a weak limit theorem for $\theta_t$ as $t\to\infty$. Later on, this result has been extended by many authors in several directions (e.g., strong limit theorems, winding numbers not only around points but also around certain curves, winding numbers for other processes), see for instance \cite{Me91}. 

In \cite{Dy62}, motivated by studying the eigenvalues of matrices from the Gaussian Unitary Ensemble, Dyson analyzed the Brownian motion in the cone formed by the Weyl chamber of type $A$, namely,
\begin{equation*}
\label{eq:def_Weyl_chamber}
     \{x=(x_1,\ldots ,x_d)\in\RR^d : x_1<\cdots <x_d\}.
\end{equation*}
He also defined the Brownian motion conditioned never to exit the chamber. These results have been extended by Biane \cite{Bi95} and Grabiner \cite{Grab99}. In \cite{Bi94}, Biane studied some further properties of the Brownian motion conditioned to stay in cones, and in particular generalized the famous Pitman's theorem to that context. In \cite{KoSc11} K\"onig and Schmid analyzed the non-exit probability \eqref{eq:def_exit_probability} of Brownian motion from a growing truncated Weyl chamber.

In \cite{Bu77}, Burkholder considered open right circular cones in any dimension and computed the values of $p>0$ such that 
\begin{equation*}
     \EE_x[\tau_C^p]<\infty. 
\end{equation*}
In \cite{DB87,DB88}, for a fairly general class of cones, DeBlassie obtained an explicit expression for the probability \eqref{eq:def_exit_probability} in terms of the eigenfunctions of the Dirichlet problem for the~Laplace-Beltrami operator on 
\begin{equation*}
\label{eq:def_Theta}
     \Theta=\mathbb S^{d-1}\cap C,
\end{equation*} 
see \cite[Theorem 1.2]{DB87}. DeBlassie also derived the asymptotics \eqref{eq:exit_asymptotic}, see \cite[Corollary 1.3]{DB87}: he found $\gamma =0$ (indeed, the drift is zero), while $\alpha$ is related to the first eigenvalue and $h(x)$ to the first eigenfunction. The basic strategy in \cite{DB87,DB88} was to show that the probability \eqref{eq:def_exit_probability} is solution to the heat equation and to solve the latter. In \cite{BaSm97}, Ba\~nuelos and Smits refined the results of DeBlassie \cite{DB87,DB88}: they considered more general cones, and obtained a quite tractable expression for the heat kernel (the transition densities for the Brownian motion in $C$ killed on the boundary), and thus for \eqref{eq:def_exit_probability}. 

We conclude this part by mentioning the work \cite{DoOC05}, in which Doumerc and O'Connell found a formula for the distribution of the first exit time of Brownian motion from a fundamental region associated with a finite reflection group.  

For Brownian motion with non-zero drift, much less is known. Only the case of Weyl chambers (of type $A$) has been investigated. In \cite{BiBoOC05}, Biane, Bougerol and O'Connell obtained an expression for the probability $\PP_x[\tau_C=\infty]=\lim_{t\to\infty}\PP_x[\tau_C>t]$ in the case where the drift is inside of the Weyl chamber (and hence the latter probability is positive). In \cite{PuRo08}, Pucha{\l}a and Rolski gave, for any drift $a$, the exact asymptotics \eqref{eq:exit_asymptotic} of the tail distribution of the exit time, in the context of Weyl chambers too. The different quantities in \eqref{eq:exit_asymptotic} were determined explicitly in terms of the drift $a$ and of a vector obtained by a procedure involving the construction of a stable partition of the drift vector.

In this article, we compute the asymptotics \eqref{eq:exit_asymptotic} for a very general class of cones $C$, and we identify $\kappa$, $h(x)$, $\alpha$ and $\gamma$ in terms of the cone $C$ and the drift $a$. 
We find that there are six different regimes depending on the position of the drift with respect to (w.r.t.)\ the cone. 

To be more specific, we will consider general cones as defined by Ba\~nuelos and Smits in~\cite{BaSm97}. Namely, given a {\em proper}, {\em open} and {\em connected} subset $\Theta$ of the unit sphere $\SS^{d-1}\subset \mathbb R^d$, we consider the cone $C$ generated by $\Theta$, that is, the set of all rays emanating from the origin and passing through $\Theta$:
\begin{equation*}
     C=\{\lambda\theta :\lambda>0, \theta\in\Theta\}.
\end{equation*}      
We associate with the cone the polar cone (which is a closed set)
\begin{equation*}
     C^{\sharp}=\{x\in\RR^d : \sclr{x}{y}\leq 0,\forall y\in C\}.
\end{equation*}  
See Figure \ref{fig:cones} for an example. Below and throughout, we shall denote by $\open{D}$ (resp.\ $\close{D}$) the interior (resp.\ the closure) of a set $D\subset \mathbb R^d$.  
The six cases leading to different regimes are then:
\begin{enumerate}[label={\rm\Alph{*}.},ref={\rm\Alph{*}}]
\item\label{case:A}polar interior drift: $a\in\open{(C^{\sharp})}$;
\item\label{case:B}zero drift: $a=0$;
\item\label{case:C}interior drift: $a\in C$;
\item\label{case:D}boundary drift: $a\in \partial C\setminus\{0\}$;
\item\label{case:E}non-polar exterior drift: $a\in \RR^d\setminus (\close{C}\cup C^{\sharp})$;
\item\label{case:F}polar boundary drift: $a\in\partial C^{\sharp}\setminus\{0\}$. 
\end{enumerate}
These cases will be analyzed in Theorems \ref{thm:case:A}, \ref{thm:case:B}, \ref{thm:case:C}, \ref{thm:case:D}, \ref{thm:case:E} and \ref{thm:case:F}, respectively. Our results show in particular that the exponential decreasing rate $e^{-\gamma}$ in \eqref{eq:exit_asymptotic} is related to the distance between the drift and the cone by the formula
\begin{equation}
\label{eq:exp_rate}
\gamma=\frac{1}{2}d(a,C)^2=\frac{1}{2}\min_{y\in \close{C}}\vert a-y\vert^2.
\end{equation}
As for the polynomial part $t^{-\alpha}$ in \eqref{eq:exit_asymptotic}, it depends on the case under consideration and reflects the local geometry of the cone at the point(s)\ that minimize the distance to the drift, plus the local geometry at the {\em contact points} between $\partial\Theta$ and the hyperplane orthogonal to the drift in case \ref{case:F}.

We would like to point out that the formula for $\gamma$ obtained in \cite{PuRo08} in the case of the Weyl chamber of type $A$ is the same as ours. Indeed, though it is not mentioned there, the vector $f$ obtained in \cite{PuRo08} via the construction of a {\em stable partition} of the drift is the projection of the drift on the Weyl chamber, and their formula (4.10) reads $\gamma=\vert a-f\vert^2/2$, as the reader can check.

\section{Assumptions on the cone and statements of results}

Though our results are stated precisely in Theorems \ref{thm:case:A}, \ref{thm:case:B}, \ref{thm:case:C}, \ref{thm:case:D}, \ref{thm:case:E} and \ref{thm:case:F}, we would like to give now a brief overview as well as precise statements.

\subsection{Assumptions on the cone}
Our main assumption on the cones studied here is the following: 
\begin{enumerate}[label={\rm(C\arabic{*})},ref={\rm(C\arabic{*})}]
\item\label{hypothesis1}The set $\Theta=\mathbb S^{d-1}\cap C$ is {\em normal}, that is, piecewise infinitely differentiable.
\end{enumerate}
With this assumption (see \cite[page 169]{Ch84}), there exists a complete set of eigenfunctions $(m_j)_{j\geq 1}$ orthonormal w.r.t.\ the surface measure on $\Theta$ with corresponding eigenvalues $0<\lambda_1<\lambda_2\leq \lambda_3\leq\cdots$, satisfying for any $j\geq 1$
\begin{equation}
\label{eq:eigenfunctions}
\begin{cases}
L_{\SS^{d-1}} m_j=-\lambda_j m_j & \mbox{on}\quad \Theta,\\
m_j=0 &\mbox{on}\quad \partial \Theta.
\end{cases}
\end{equation} 
where $L_{\SS^{d-1}}$ denotes the Laplace-Beltrami operator on $\SS^{d-1}$.
We shall say that the cone is normal if $\Theta$ is normal.
For any $j\geq 1$, we set 
\begin{equation}
\label{eq:alpha_j}
     \alpha_j=\sqrt{\lambda_j+({d}/{2}-1)^2}
\end{equation}
and 
\begin{equation}
\label{eq:p_j}
     p_j=\alpha_j-({d}/{2}-1).
\end{equation}

\begin{example}
\label{ex:dim2}
In dimension $2$, any (connected and proper) open cone is a rotation of 
\begin{equation*}
\label{eq:any_cone}
     \{\rho e^{i\theta}: \rho> 0, 0<\theta < \beta\}
\end{equation*}
for some $\beta\in(0,2\pi]$, see Figure \ref{fig:cones}.
A direct computation starting from Equation \eqref{eq:eigenfunctions} yields $\lambda_j = (j\pi/\beta)^2$, and thus 
\begin{equation*}
     p_j = \alpha_j = j\pi/\beta,
\end{equation*}
for any $j\geq 1$. Further, the eigenfunctions \eqref{eq:eigenfunctions} are given in polar coordinates by
\begin{equation}
\label{eq:expression_eigenfunctions_2}
     m_j(\theta)=\frac{2}{\beta}\sin\left(\frac{j\pi \theta}{\beta}\right),\quad \forall j\geq 1,
\end{equation}
where the term $2/\beta$ comes from the normalization $\int_0^\beta m_j(\theta)^2\text{d}\theta =1$. 
\end{example}

\unitlength=0.6cm
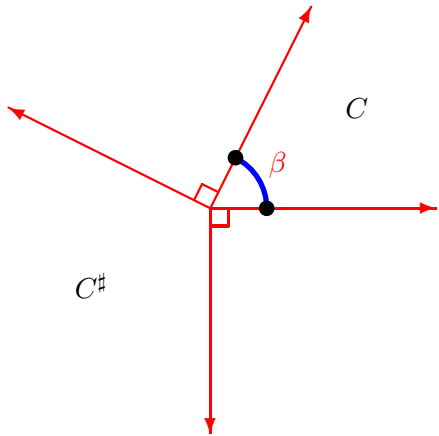
\begin{figure}[t]
\begin{center}
\begin{tabular}{ccccc}
    \begin{picture}(0,5)
    \thicklines
    \put(0,0){\textcolor{red}{\vector(1,0){5}}}
    \put(0,0){\textcolor{red}{\vector(0,-1){5}}}
    \put(0.4,-0.4){\textcolor{red}{\line(0,1){0.4}}}
    \put(0,-0.4){\textcolor{red}{\line(1,0){0.4}}}
    \put(0,0){\textcolor{red}{\vector(1,2){2.236}}}
    \put(0,0){\textcolor{red}{\vector(-2,1){4.472}}}
    \put(-0.358,0.179){\textcolor{red}{\line(1,2){0.179}}}
    \put(0.179,0.358){\textcolor{red}{\line(-2,1){0.36}}}
    \put(3,2){$C$}
    \put(-3,-2){$C^\sharp$}
    \psarc[linecolor=blue,linewidth=2pt](0,0){0.75}{0}{63.44}
    {\put(1.25,0){\textcolor{black}{\circle*{0.3}}}}
    {\put(.559,1.117){\textcolor{black}{\circle*{0.3}}}}
    \put(1.3,0.8){\textcolor{red}{$\beta$}}
    \end{picture}
    \end{tabular}
\end{center}
\vspace{25mm}
\caption{Cones $C$ with opening angle $\beta$ and polar cones $C^\sharp$ in dimension $2$. The set $\Theta$ (the arc of circle) and its boundary are particularly important in our analysis.}
\label{fig:cones}
\end{figure}

The functions $m_j$ and constants $\alpha_j$ are particularly important in this study because they allow to write a series expansion for the heat kernel of the cone (Lemma~\ref{lemma:heat_kernel}) to which the non-exit probability is explicitly related (Lemma~\ref{lemma:expression_heat_kernel}).

Cases \ref{case:A}, \ref{case:B} and \ref{case:C} are treated with full generality under the sole assumption \ref{hypothesis1}. Thus we extend the corresponding 
results of Pucha{\l}a and Rolski in \cite{PuRo08} about Weyl chambers of type $A$ in these cases. (Note that case \ref{case:A} is new since the polar cone of a Weyl chamber of type $A$ has an empty interior, whereas case \ref{case:B} has already been settled in \cite{BaSm97}, but is presented here for the sake of completeness.) 

Cases \ref{case:D}, \ref{case:E} and \ref{case:F} will be considered under an additional smoothness assumption on the cone that excludes Weyl chambers from our analysis. The reason is that we will need estimates for the heat kernel of the cone at boundary points, and those are only available (to our knowledge)\ in the case of smooth cones or, on the other hand, in the case of Weyl chambers. 
More precisely, we shall assume in these cases that:
\begin{enumerate}[label={\rm(C\arabic{*})},ref={\rm(C\arabic{*})}]
\setcounter{enumi}{1}
\item\label{hypothesis2}The set $\Theta=\mathbb S^{d-1}\cap C$ is real-analytic.\footnote{A domain $\Omega\subset \mathbb R^d$ is real-analytic if at each point $x\in \partial\Omega$ there is a ball $B(x,r)$ with $r>0$ and a one-to-one mapping $\psi$ of $B(x,r)$ onto a certain domain $D\subset \mathbb R^d$ such that (i) $\phi(B(x,r)\cap \Omega)\subset [0,\infty)^d$, (ii) $\phi(B(x,r)\cap \partial\Omega)\subset \partial([0,\infty)^d)$, (iii) $\psi$ and $\psi^{-1}$ are real-analytic functions on $B(x,r)$ and $D$, respectively. This is equivalent to the fact that each point of $\partial\Omega$ has a neighborhood in which $\partial\Omega$ is the graph of a real-analytic function of $n-1$ coordinates.
We refer to \cite[section 6.2]{GiTr83} for more details.\label{footnote:real-analytic}}
\end{enumerate}
Notice that under this assumption $\Theta$ is normal (in other words, \ref{hypothesis1} implies \ref{hypothesis2}).

We have already mentioned the formula for the exponential decreasing rate:
\begin{equation*}
     \gamma=\frac{1}{2}d(a,C)^2,
\end{equation*}     
and the reader can already imagine the importance of the set
\begin{equation*}
     \Pi(a)=\{y\in\close{C}: \vert a-y\vert=d(a,C)\}.
\end{equation*}         
Indeed, the formula for the non-exit probability involves an integral of Laplace's type, and only neighborhoods of the points of $\Pi(a)$ will contribute to the asymptotics. It follows by elementary topological arguments that $\Pi(a)$ is a non-empty compact set. 
In cases \ref{case:A}, \ref{case:B}, \ref{case:C}, \ref{case:D} and \ref{case:F}, this set is a singleton ($\{0\}$ or $\{a\}$ according to the case), but in case \ref{case:E} it may have infinitely many points. Since we are not able to handle the case where $\Pi(a)$ has an accumulation point, we shall assume (in case \ref{case:E} only) that
\begin{enumerate}[label={\rm(C\arabic{*})},ref={\rm(C\arabic{*})}]
\setcounter{enumi}{2}
\item\label{hypothesis3}The set $\Pi(a)$ is finite.
\end{enumerate}
This holds if the cone is convex for example.

Our final comment concerns the case \ref{case:F}. Surprisingly, it is the most difficult: it is a mixture between cases \ref{case:A} and \ref{case:B}, and its analysis reveals an unexpected (at first sight) contribution of the contact points (see section \ref{subsec:case:F} for a precise definition) between $\partial\Theta$ and the hyperplane orthogonal to the drift. Here again, we shall add a technical assumption, namely:
\begin{enumerate}[label={\rm(C\arabic{*})},ref={\rm(C\arabic{*})}]
\setcounter{enumi}{3}
\item\label{hypothesis4}The set of contact points $\Theta_c$ is finite.
\end{enumerate}
Moreover, we will consider case \ref{case:F} only in dimension $2$ (where \ref{hypothesis4} always holds) and $3$. The reason is that we are technically not able to handle more general cases.

\subsection{Main results}

The following theorem summarizes our results. Some important comments may be found below.

\begin{theorem-sum}
\label{thm:generic_theorem}Let $C$ be a normal cone in $\RR^d$ {\rm(}hypothesis \ref{hypothesis1}{\rm)}. For Brownian motion with drift $a$, in each of the six cases \ref{case:A}, \ref{case:B}, \ref{case:C}, \ref{case:D}, \ref{case:E} and \ref{case:F}, 
the asymptotic behavior of the non-exit probability is given by
\begin{equation*}
     \PP_x[\tau_C>t]=\kappa h(x) t^{-\alpha} e^{-\gamma t}(1+o(1)),\quad t\to\infty,
\end{equation*}
where
\begin{equation*}
     \gamma=\frac{1}{2}d(a,C)^2,
\end{equation*}
and
\begin{equation*}
 \alpha=
\begin{cases}
\alpha_1+1 & \mbox{if $a$ is a polar interior drift {\rm(}case \ref{case:A}{\rm)},}\\
p_1/2 & \mbox{if $a=0$ {\rm(}case \ref{case:B}{\rm)},}\\
0 & \mbox{if $a$ is an interior drift {\rm(}case \ref{case:C}{\rm)},}\\
1/2 & \mbox{if $a$ is a boundary drift {\rm(}case \ref{case:D}{\rm)} and $\Theta$ is real-analytic \ref{hypothesis2},}\\
3/2 & \mbox{if $a$ is a non-polar exterior drift {\rm(}case \ref{case:E}{\rm)}, \ref{hypothesis2} and \ref{hypothesis3},}\\
p_1/2+1 & \mbox{if $a$ is a polar boundary drift {\rm(}case \ref{case:F}{\rm)} and $C$ is two-dimensional.}
\end{cases}
\end{equation*}
\end{theorem-sum}

The constants $\kappa$ and the functions $h(x)$ are also explicit, but their expression is rather complicated in some cases. For this reason they are given in the corresponding sections. As a matter of example, let us give them in case \ref{case:A}:
\begin{equation*}
     \kappa_A=\frac{1}{2^{\alpha_1}\Gamma(\alpha_1+1)}\int_C e^{\sclr{a}{y}}\vert y\vert^{p_1}m_1(\vec{y}) \text{d}y,\quad h_A(x)=e^{\sclr{-a}{x}}\vert x\vert^{p_1}m_1(\vec{x}),
\end{equation*}
where $m_1$ is defined in \eqref{eq:eigenfunctions} and $\alpha_1$ in \eqref{eq:alpha_j}, and where for any $y\not=0$, we denote by $\vec{y}=y/\vert y\vert$ its projection on the unit sphere $\SS^{d-1}$. Above, case \ref{case:F} is presented in dimension $2$ only, because the value of $\alpha$ in dimension $3$ is quite complicated (we refer to Theorem \ref{thm:case:F3} for the full statement).

\section{The example of two-dimensional Brownian motion in cones}

For the one-dimensional Brownian motion and the cone $C=(0,\infty)$, there are three regimes for the asymptotics of the non-exit probability, according to the sign of the drift $a\in\mathbb R$. Precisely, for any $x>0$, as $t\to\infty$ one has, with obvious notations (see \cite[section 2.8]{KaSh91}),
\begin{equation}
\label{eq:asymptotic_1D}
     \PP_x[\tau_{(0,\infty)}>t] = (1+o(1))\left\{\begin{array}{ccc}
     \displaystyle\frac{xe^{-a x} e^{-ta^2/2}}{\sqrt{2\pi} a^2 t^{3/2}} & \text{if} & a<0,\vspace{1mm}\\
     \displaystyle\frac{\sqrt{2}x}{\sqrt{\pi t}} & \text{if} & a=0,\vspace{1mm}\\
     \displaystyle 1-e^{-2ax}& \text{if} & a>0.
     \end{array}\right.
\end{equation}

For some specific two-dimensional cones, the asymptotics of the non-exit probability is easy to determine. This is for example the case of the upper half-plane since this is essentially a one-dimensional case. It is also an easy task to deal with the quarter plane $Q$. Indeed, by independence of the coordinates $(B_t^{(1)},B_t^{(2)})$ of the Brownian motion $B_t$, the non-exit probability can be written as the product
\begin{equation*}
     \mathbb P_{x}[\tau_Q>t] = \mathbb P_{x_1}[\tau_{(0,\infty)}(B^{(1)})>t]\cdot\mathbb P_{x_2}[\tau_{(0,\infty)}(B^{(2)})>t],
\end{equation*}
where $x=(x_1,x_2)$.
Denoting by $a=(a_1,a_2)$ the coordinates of the drift and making use of \eqref{eq:asymptotic_1D}, one readily deduces the asymptotics $\mathbb P_x[\tau_Q>t]=\kappa h(x)t^{-\alpha}e^{-\gamma t}(1+o(1))$, as summarized in Figure \ref{fig:quarter_plane}, where the value of $\alpha$ is given, according to the position of the drift $(a_1,a_2)$ in the quarter plane. We focus on $\alpha$ and not on $\gamma$, since the value of $\gamma$ is always obtained in the same way.

\unitlength=0.6cm
\begin{figure}[h]
\begin{center}
\begin{tabular}{ccccc}
    \begin{picture}(0,5)
    \thicklines
    \put(-5,0){\textcolor{red}{\line(1,0){4}}}
    \put(0,-5){\textcolor{red}{\line(0,1){1}}}
    \put(1,0){\textcolor{red}{\line(1,0){1}}}
    \put(4,0){\textcolor{red}{\vector(1,0){1}}}
    \put(0,1){\textcolor{red}{\vector(0,1){4}}}
    \put(0,-2){\textcolor{red}{\line(0,1){1}}}
    \put(-0.15,-0.16){$1$}
    \put(2.85,2.84){$0$}
    \put(-0.15,-3.16){$2$}
    \put(-3.15,-3.16){$3$}
    \put(2.55,-0.16){$1/2$}
    \put(2.55,-3.16){$3/2$}
    \put(5.15,-0.15){$a_1$}
    \put(-0.15,5.15){$a_2$}
    {\put(0,-3){\textcolor{blue}{\circle{2}}}}
    {\put(3,-3){\textcolor{blue}{\circle{2}}}}
    {\put(-3,-3){\textcolor{blue}{\circle{2}}}}
    \put(3,3){\textcolor{blue}{\circle{2}}}
     {\put(0,0){\textcolor{blue}{\circle{2}}}}
     {\put(3,0){\textcolor{blue}{\circle{2}}}}
    \end{picture}
    \end{tabular}
\end{center}
\vspace{25mm}
\caption{Value of $\alpha$ in terms of the position of the drift $(a_1,a_2)$ in the plane (case of the quarter plane)}
\label{fig:quarter_plane}
\end{figure}
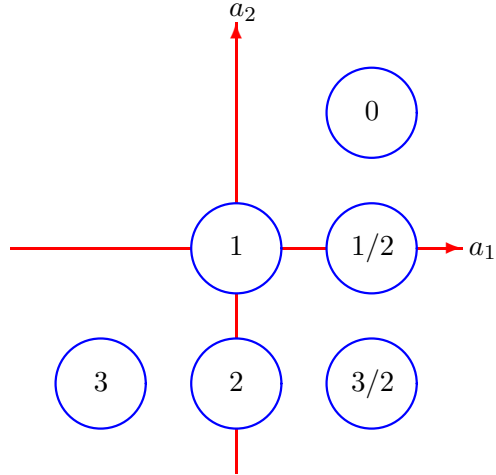

More generally, our results show that the value of $\alpha$ for any two-dimensional cone is given as in Figure \ref{fig:cone}.
\unitlength=0.6cm
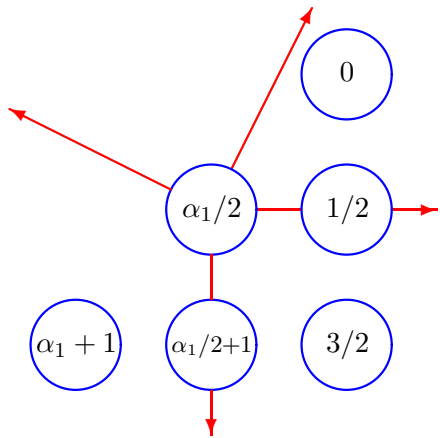
\begin{figure}[h]
\begin{center}
\begin{tabular}{ccccc}
    \begin{picture}(0,5)
    \thicklines
    \put(1,0){\textcolor{red}{\line(1,0){1}}}
    \put(4,0){\textcolor{red}{\vector(1,0){1}}}
    \put(.447,.894){\textcolor{red}{\vector(1,2){1.79}}}
    \put(-0.65,-0.16){$\alpha_1/2$}
    \put(2.85,2.84){$0$}
    \put(-0.88,-3.16){\footnotesize{${\alpha_1}/{2}\hspace{-0.8mm}+\hspace{-0.8mm}1$}}
    \put(-3.88,-3.16){$\alpha_1+1$}
    \put(2.55,-0.16){$1/2$}
    \put(2.55,-3.16){$3/2$}
    {\put(0,-3){\textcolor{blue}{\circle{2}}}}
    {\put(3,-3){\textcolor{blue}{\circle{2}}}}
    {\put(-3,-3){\textcolor{blue}{\circle{2}}}}
    \put(3,3){\textcolor{blue}{\circle{2}}}
     {\put(0,0){\textcolor{blue}{\circle{2}}}}
     {\put(3,0){\textcolor{blue}{\circle{2}}}}
     \put(0,-4){\textcolor{red}{\vector(0,-1){1}}}
     \put(-.894,0.447){\textcolor{red}{\vector(-2,1){3.58}}}
     \put(0,-2){\textcolor{red}{\line(0,1){1}}}
    \end{picture}
    \end{tabular}
\end{center}
\vspace{25mm}
\caption{Value of $\alpha$ in terms of the position of the drift $(a_1,a_2)$ in the plane (case of a general cone of opening angle $\beta$, for which $\alpha_1=\pi/\beta$, see Figure \ref{fig:cones})}
\label{fig:cone}
\end{figure}
This can be understood as follows: when the drift is negative (i.e., when it belongs to the polar cone $C^\sharp$), one sees the influence of the vertex of the cone ($\alpha$ is expressed with the opening angle $\beta$) since the trajectories that do not leave the cone will typically stay close to the origin. In all other cases, the Brownian motion will move away from the vertex, and will see the cone as a half-space (boundary drift and non-polar exterior drift) or as a whole-space (interior drift).

\section{Preliminary results}

\label{sec:heat_kernel_cone}

In this section we introduce all necessary tools for our study. We first give the expression of the non-exit probability \eqref{eq:def_exit_probability} in terms of the heat kernel of the cone $C$ (see Lemmas \ref{lemma:expression_heat_kernel} and \ref{lem:non_exit_expression}). Then we guess the value of the exponential decreasing rate of this probability, by simple considerations on its integral expression. Finally we present our general strategy to compute the asymptotics of the non-exit probability.

\subsection{Expression of the non-exit probability}

In what follows we consider $(B_t)_{t\geq0}$ a $d$-dimensional Brownian motion with drift $a$ and identity covariance matrix. Under $\PP_x$, the Brownian motion starts at $x\in\mathbb R^d$.

The lemma hereafter gives an expression of the non-exit probability for Brownian motion with drift $a$ in terms of an integral involving the transition probabilities of the Brownian motion with zero drift killed at the boundary of the cone. This is a quite standard result (see \cite[Proposition 2.2]{PuRo08} for example) and an easy consequence of Girsanov theorem.
Notice that this result is not at all specific to cones and is valid for any domain in $\RR^d$.

\begin{lemma}
\label{lemma:expression_heat_kernel}
Let $p^C(t,x,y)$ denote the transition probabilities of the Brownian motion with zero drift killed at the boundary of the cone $C$. We have
\begin{equation}
\label{exittime}
     \PP_x[\tau_C>t]=e^{\sclr{-a}{x}-t\vert a\vert^2/2}\int_C e^{\sclr{a}{y}}p^C(t,x,y) \textnormal{d}y,\quad \forall t\geq 0.
\end{equation}
\end{lemma}

We shall now write a series expansion for the transition probabilities of the Brownian motion killed at the boundary of $C$ (or equivalently, see \cite[section 4]{Hu56}, for the heat kernel $p^C(t,x,y)$ of the cone $C$), as given in~\cite{BaSm97}. We denote by $I_{\nu}$ the modified Bessel function of order $\nu$:
\begin{equation}
\label{bessel}
     I_{\nu}(x)=\frac{2({x}/{2})^{\nu}}{\sqrt{\pi}\Gamma(\nu+1/2)}\int_{0}^{\frac{\pi}{2}}(\sin t)^{2\nu}\cosh(x\cos t)\text{d}t=\sum_{m=0}^{\infty}\frac{({x}/{2})^{\nu+2m}}{m!\Gamma(\nu+m+1)}.
\end{equation}
It satisfies the second order differential equation
\begin{equation*}
     I_\nu ''(x)+\frac{1}{x}I_\nu'(x)=\left(1+\frac{\nu^2}{x^2}\right) I_\nu(x).
\end{equation*} 
Its leading asymptotic behavior near $0$ is given by:
\begin{equation}
\label{eq:Bessel_equivalent_0}
     I_{\nu}(x)=\frac{x^{\nu}}{2^{\nu}\Gamma(\nu+1)}(1+o(1)), \quad x\to 0.
\end{equation}   
We refer to \cite{Wa44} for proofs of the facts above and for any further result.

\begin{lemma}[\cite{BaSm97}] 
\label{lemma:heat_kernel}
Under \ref{hypothesis1}, the heat kernel of the cone $C$ has the series expansion
\begin{equation}
\label{heatkernel}
     p^C(t,x,y)=\frac{e^{-\frac{\vert x\vert^2 +\vert y\vert^2}{2t}}}{t(\vert x\vert\vert y\vert)^{{d}/{2}-1}}\sum_{j=1}^{\infty}I_{\alpha_j}\left(\frac{\vert x\vert \vert y\vert}{t}\right)m_j(\vec{x})m_j(\vec{y}),
\end{equation}
where the convergence is uniform for $(t,x,y)\in [T,\infty)\times\{x\in C:\vert x\vert\leq R\}\times C$, for any positive constants $T$ and $R$.
\end{lemma}

Making the change of variables $y\mapsto ty$ in \eqref{exittime} and using \eqref{heatkernel}, we easily obtain the following lemma, where the expression of the non-exit probability now involves an integral of Laplace's type.
\begin{lemma}
\label{lem:non_exit_expression}
Let $C$ be a normal cone. For Brownian motion with drift $a$, the non-exit probability is given by
\begin{equation}
\label{exittime_after_change}
     \PP_x[\tau_C>t]=e^{\sclr{-a}{x}-\vert x\vert^2/(2t)+\vert x\vert^2/2}t^{d/2}\int_C e^{\vert y\vert^2/2}p^C(1,x,y) e^{-t\vert a-y\vert^2/2}\textnormal{d}y,\quad \forall t\geq 0.
\end{equation}
\end{lemma}

\subsection{General strategy}
The aim now is to understand the asymptotic behavior as $t\to\infty$ of the integral in the right-hand side of \eqref{exittime_after_change}. 
First, we notice that it suffices to analyze the asymptotic behavior of
\begin{equation}
\label{eq:definition_of_I}
     I(t)=t^{d/2}\int_C e^{\vert y\vert^2/2}p^C(1,x,y) e^{-t\vert a-y\vert^2/2}\text{d}y.
\end{equation}
To do this, we shall use  Laplace's method \cite[Chapter 5]{Co65}. The basic question when applying this method is to locate the points $y\in \close{C}$ where the function 
\begin{equation*}
\label{eq:function_to_max}
   \vert a-y\vert^2/2
\end{equation*}
in the exponential reaches its minimum value,
for it is expected that only a neighborhood of these points will contribute to the {asymptotics}.
And indeed, we shall prove that the exponential decreasing rate $e^{-\gamma}$ of the non-exit probability in \eqref{eq:exit_asymptotic} is given, for the six cases \ref{case:A}--\ref{case:F}, by \eqref{eq:exp_rate}, namely
\begin{equation*}
    \gamma=\frac{1}{2}\min_{y\in \close{C}}\vert a-y\vert^2=\frac{1}{2}d(a,C)^2.
\end{equation*}
Specifically, let $\Pi(a)$ be the set of minimum points, that is,
\begin{equation*}
     \Pi(a)=\{y\in\close{C}: \vert a-y\vert=d(a,C)\}.
\end{equation*}     
It follows by elementary topological arguments that $\Pi(a)$ is a non-empty compact set. 
The lemma below shows that if the domain of integration is restricted to the complement of any neighborhood of $\Pi(a)$, then the integral in \eqref{eq:definition_of_I} becomes negligible w.r.t.\ the expected exponential rate $e^{-t\gamma}$. To be precise, consider the open $\delta$-neighborhood of $\Pi(a)$:
\begin{equation*}
     \Pi_{\delta}(a)=\{y\in\mathbb R^d: d(y,\Pi(a))<\delta\}.
\end{equation*}     
\begin{lemma}
\label{lem:bound_for_non_contributive_integral}
For any $\delta>0$, there exists $\eta>0$ such that
\begin{equation*}
\int_{C\setminus\Pi_{\delta}(a)}e^{\vert y\vert^2/2}p^C(1,x,y)e^{-t\vert a-y\vert^2/2} \textnormal{d}y=O(e^{-t(\gamma+\eta)}),\quad t\to\infty,
\end{equation*}
where $\gamma$ is the quantity defined in \eqref{eq:exp_rate}.
\end{lemma}

\begin{proof} 
Let $\delta>0$ and define
\begin{equation*}
J_\delta(t)=\int_{C\setminus\Pi_{\delta}(a)}e^{\vert y\vert^2/2}p^C(1,x,y)e^{-t\vert a-y\vert^2/2} \text{d}y.
\end{equation*}
From the inequality
$
\vert y\vert^2\leq (\vert y-a\vert+\vert a\vert)^2\leq 2\vert y-a\vert^2+2\vert a\vert^2
$,
we obtain the upper bound
$
e^{\vert y\vert^2/2}\leq ce^{2\vert y-a\vert^2}
$,
from which we deduce that
\begin{equation*}
0\leq J_\delta(t)\leq c \int_{C\setminus\Pi_{\delta}(a)}p^C(1,x,y)e^{-s\vert a-y\vert^2/2} \text{d}y,
\end{equation*}
where $s=t-2$. Since $y\mapsto \vert a-y\vert^2/2$ is coercive and continuous, its infimum on the closed set $\close{C}\setminus\Pi_{\delta}(a)$ is a minimum.
Thus, by definition of $\Pi(a)$, we have
\begin{equation*}
     \inf_{\close{C}\setminus\Pi_{\delta}(a)}\vert a-y\vert^2/2>\gamma.
\end{equation*}     
In other words, there exists $\eta>0$
such that $\vert a-y\vert^2/2\geq \gamma+\eta$ on $\close{C}\setminus\Pi_{\delta}(a)$.
Hence, for all $s\geq 0$, we have
\begin{equation*}
0\leq J_\delta(t)\leq c e^{-s(\gamma+\eta)} \int_{C\setminus\Pi_{\delta}(a)}p^C(1,x,y)\text{d}y\leq c e^{-s(\gamma+\eta)}.
\end{equation*}
This concludes the proof of the lemma.
\end{proof}

It is now clear that the strategy to analyze the non-exit probability is to determine the asymptotic behavior of the integral $I_{\delta}(t)$, which is defined by
\begin{equation}
\label{eq:defintion_of_I_delta}
I_{\delta}(t)=t^{d/2}\int_{C\cap\Pi_{\delta}(a)}e^{\vert y\vert^2/2}p^C(1,x,y)e^{-t\vert a-y\vert^2/2} \text{d}y,
\end{equation}
and to check that it has the right exponential decreasing rate $e^{-\gamma}$, as expected. Indeed, in this case, the asymptotic behavior of $I(t)$, and consequently that of the non-exit probability, can be derived from the asymptotics of $I_{\delta}(t)$, as explained in the next lemma, which will constitute our general proof strategy.
\begin{lemma}
\label{lem:general_proof_strategy}
Suppose that $g(t)$ is a function satisfying conditions \ref{fit} and \ref{sit} below:
\begin{enumerate}[label={\rm(\roman{*})},ref={\rm(\roman{*})}]
\item\label{fit}$g(t)=\kappa t^{-\alpha}e^{-t\gamma}$ for some $\kappa>0$ and $\alpha\in\RR$;
\item\label{sit}For all $\epsilon>0$, there exists $\delta>0$ such that
\begin{equation*}
1-\epsilon\leq\liminf_{t\to\infty}\frac{I_{\delta}(t)}{g(t)}\leq\limsup_{t\to\infty}\frac{I_{\delta}(t)}{g(t)}\leq 1+\epsilon.
\end{equation*}
\end{enumerate}
Then $I(t)= g(t)(1+o(1))$ as $t\to\infty$.
\end{lemma}
\begin{proof} 
It follows from Lemma \ref{lem:bound_for_non_contributive_integral} as an easy exercise.
\end{proof}

In our study of $I_\delta(t)$, it will be important that the elements of $\Pi(a)$ be isolated from each other. By compactness, this condition is equivalent to the fact that $\Pi(a)$ be finite. 
In that case, for $\delta>0$ small enough, $I_{\delta}(t)$ decomposes into the finite sum
$$I_{\delta}(t)=t^{d/2}\sum_{p\in \Pi(a)}\int_{C\cap B(p,\delta)}e^{\vert y\vert^2/2}p^C(1,x,y)e^{-t\vert a-y\vert^2/2} \text{d}y,$$
where $B(p,\delta)$ does not contain any other minimum point than $p$. The contribution of each minimum point $p$ can then be analyzed separately.
The reason to do that is that we simply don't know how to handle the general case.

In most cases, it is not much of a restriction.
Indeed, for a convex cone (or any convex set), the set $\Pi(a)$ reduces to a single point, namely the projection $p_C(a)$ of $a$ on $\close{C}$. Though the projection may not be unique in general (that is, when the cone is not convex), it is still true
in cases \ref{case:A}, \ref{case:B}, \ref{case:C}, \ref{case:D} and \ref{case:F} that $\Pi(a)$ has only one element, namely $p=0$
(cases \ref{case:A}, \ref{case:B}, \ref{case:F}) or $p=a$ (cases \ref{case:C} and \ref{case:D}), and that this point satisfies the usual property $\sclr{a-p}{y-p}\leq 0$ for all $y\in\close{C}$. Therefore, we call this point {\em the} projection and write it $p_C(a)$.
The condition that $\Pi(a)$ be finite is a restriction only in case \ref{case:E}: according to the cone, the minimum could be reached at infinitely many different points, but we leave this general setting as an open problem.

\section{Precise statements and proofs of the theorems \ref{case:A}--\ref{case:F}}

\subsection{Case \ref{case:A} (polar interior drift)}
In this section, we study the case where the drift $a$ belongs to the interior of the polar cone $C^\sharp$. It might be thought of as the natural generalization of the one-dimensional negative drift case. Define (with $p_1$ as in \eqref{eq:p_j})
\begin{equation}
\label{eq:harmonic_Brownian}
     u(x)=\vert x\vert^{p_1}m_1(\vec{x}).
\end{equation} 
The function $u$ is the unique (up to multiplicative constants) positive harmonic function of Brownian motion killed at the boundary of $C$. We also define (with $\alpha_1$ as in \eqref{eq:alpha_j})
\begin{equation*}
     \kappa_A=\frac{1}{2^{\alpha_1}\Gamma(\alpha_1+1)}\int_C e^{\sclr{a}{y}}u(y) \text{d}y,
\end{equation*}
as well as
\begin{equation*}     
     h_A(x)=e^{\sclr{-a}{x}}u(x).
\end{equation*}
Notice that $\kappa_A$ is finite because $a\in \open{(C^\sharp)}$ (see Lemma \ref{polarconeinterior}). Our main result in this section is the following:
\begin{theorem}
\label{thm:case:A}
Let $C$ be a normal cone. If the drift $a$ belongs to the interior of the polar cone $C^{\sharp}$, then
\begin{equation*}
     \PP_x[\tau_C>t]= \kappa_A h_A(x)t^{-(\alpha_1+1)}e^{-t\vert a\vert^2/2}(1+o(1)),\quad t\to\infty.
\end{equation*}
\end{theorem}

\begin{proof} 
Since $a\in \open{(C^\sharp)}$, the projection $p_C(a)$  is $0$ and $\gamma=\vert a\vert^2/2$. According to our general strategy, we focus our attention on
\begin{equation*}
\label{eq:starting_formula_case_A}
     I_\delta(t)=t^{d/2}\int_{\{y\in C:\vert y \vert\leq\delta\}} e^{\vert y\vert^2/2}p^C(1,x,y) e^{-t\vert a-y\vert^2/2}\text{d}y.
\end{equation*}
Let $\epsilon>0$ be given. It follows from Lemma~\ref{equibrotue} below that there exists $\delta>0$ such that $p^C(1,x,y)$ is bounded from above and below 
on $\{y\in C: \vert y\vert\leq \delta\}$ by
\begin{equation*}
     (1\pm\epsilon) bu(x)u(y)e^{-(\vert x\vert^2+\vert y\vert^2)/2},
\end{equation*}
where $b=(2^{\alpha_1}\Gamma(\alpha_1+1))^{-1}$. 
Therefore,  $I_\delta(t)$ is bounded from above and below by
\begin{equation}
\label{eq:bound_for_I_delta_case_A}
(1\pm\epsilon) bu(x)e^{-\vert x\vert^2/2}t^{d/2}\int_{\{y\in C: \vert y\vert\leq \delta\}} u(y) e^{-t\vert a-y\vert^2/2}\text{d}y.
\end{equation}
By making the change of variables $v=ty$ and using the homogeneity of $u$, this expression becomes
\begin{equation*}(1\pm\epsilon) bu(x)e^{-\vert x\vert^2/2}t^{-(\alpha_1+1)}e^{-t\vert a\vert^2/2}\int_{\{v\in C: \vert v\vert\leq t\delta\}} u(v) e^{\sclr{a}{v}-\vert v\vert^2/(2t)}\text{d}v.
\end{equation*}
Now, since $a\in \open{(C^\sharp)}$ implies that $\sclr{a}{v}\leq -c\vert v\vert$ for all $v\in C$, for some $c>0$ (see Lemma~\ref{polarconeinterior} below), 
the function $u(v)e^{\sclr{a}{v}}$ is integrable on $C$. Therefore, we can apply the dominated convergence theorem to obtain
\begin{equation*}
\int_{\{v\in C: \vert v\vert\leq t\delta\}} u(v) e^{\sclr{a}{v}-\vert v\vert^2/(2t)}\text{d}v=(1+o_\delta(1))\int_{C} u(v) e^{\sclr{a}{v}}\text{d}v, \quad t\to\infty.
\end{equation*}
Hence, the bound for $I_\delta(t)$ can finally be written as
\begin{equation*}
(1\pm\epsilon) \kappa_A u(x)e^{-\vert x\vert^2/2}t^{-(\alpha_1+1)}e^{-t\vert a\vert^2/2}(1+o_\delta(1)), \quad t\to\infty,
\end{equation*}
and a direct application of Lemma~\ref{lem:general_proof_strategy} gives
\begin{equation*}
I(t)= \kappa_A u(x)e^{-\vert x\vert^2/2}t^{-(\alpha_1+1)}e^{-t\vert a\vert^2/2}(1+o(1)), \quad t\to\infty.
\end{equation*}
The theorem then follows thanks to the expression \eqref{exittime_after_change} of the non-exit probability.
\end{proof}

We now state and prove a lemma that was used in the proof of Theorem \ref{thm:case:A}. Similar estimates can be found in \cite[section 5]{Ga09}.

\begin{lemma}
\label{equibrotue}
We have
\begin{equation*}
     \lim_{\vert y\vert\to 0}\frac{p^C(1,x,y)e^{(\vert x\vert^2+\vert y\vert^2)/2}}{u(x)u(y)}=(2^{\alpha_1}\Gamma(\alpha_1+1))^{-1}
\end{equation*}
uniformly on $\{x\in C:\vert x\vert\leq R\}$, for any positive constant $R$.
\end{lemma}

\begin{proof} 
For brevity, let us write $x=\rho\theta$ and $y=r\eta$, with $\rho, r>0$ and $\theta, \eta\in \Theta$, and set $M=\rho r$. It follows from the expression of the heat kernel \eqref{heatkernel} that
\begin{equation*}
     \frac{p^C(1,\rho \theta,r\eta)e^{(\rho^2+r^2)/2}}{u(\rho\theta)u(r\eta)}
     =
     \sum_{j=1}^{\infty}\frac{I_{\alpha_j}(M)}{M^{\alpha_1}}
\frac{m_j(\theta)}{m_1(\theta)}\frac{m_j(\eta)}{m_1(\eta)}.
\end{equation*}
Using then equation \eqref{compfoncprop} from Lemma \ref{lemma_BaSm97} below, we find the upper bound (below and throughout, $c$ will denote a positive constant, possibly depending on the dimension $d$, which can take different values from line to line)
\begin{equation}
\label{eq:upper_bound}
     \left\vert \frac{I_{\alpha_j}(M)}{M^{\alpha_1}}\frac{m_j(\theta)}{m_1(\theta)}\frac{m_j(\eta)}{m_1(\eta)}\right\vert
\leq\frac{c}{M^{\alpha_1}}\frac{I_{\alpha_j}(M)}{I_{\alpha_j}(1)}.
\end{equation}
Now, using the integral expression \eqref{bessel} for $I_{\alpha_j}$, we obtain
\begin{align*}
\label{eq:two_estimations}
     I_{\alpha_j}(M)\leq & \frac{2\left(\frac{M}{2}\right)^{\alpha_j}}{\sqrt{\pi}\Gamma(\alpha_j+{1}/{2})}\cosh(M)
\int_{0}^{\frac{\pi}{2}}(\sin t)^{2\alpha_j}\text{d}t,\\
 I_{\alpha_j}(1)\geq & \frac{2\left(\frac{1}{2}\right)^{\alpha_j}}{\sqrt{\pi}\Gamma(\alpha_j+1/2)}
\int_{0}^{\frac{\pi}{2}}(\sin t)^{2\alpha_j}\text{d}t.\nonumber
\end{align*}
We conclude that
\begin{equation*}
     \frac{I_{\alpha_j}(M)}{I_{\alpha_j}(1)}\leq M^{\alpha_j}\cosh(M).
\end{equation*}     
Using the latter estimation in \eqref{eq:upper_bound}, we deduce that
\begin{equation*}
\label{eq:upper_bound_d}
     \left\vert \frac{I_{\alpha_j}(M)}{M^{\alpha_1}}\frac{m_j(\theta)}{m_1(\theta)}\frac{m_j(\eta)}{m_1(\eta)}\right\vert
\leq c M^{\alpha_j-\alpha_1}\cosh(M).
\end{equation*}
It is easily seen from equation \eqref{croissalp} in Lemma \ref{lemma_BaSm97} below that $\sum_{j=1}^{\infty}M^{\alpha_j-\alpha_1}\cosh(M)$ is a uniformly convergent series for $M\in[0,1-\epsilon]$, for any $\epsilon\in (0,1]$. This implies that the series
\begin{equation*}
     \sum_{j=1}^\infty \frac{I_{\alpha_j}(M)}{M^{\alpha_1}}\frac{m_j(\theta)}{m_1(\theta)}\frac{m_j(\eta)}{m_1(\eta)}
\end{equation*}     
is uniformly convergent for $(M,\theta,\eta)\in [0,1-\epsilon]\times \Theta\times \Theta$, for any $\epsilon\in (0,1]$. Therefore we can take the limit term by term. Since
\begin{equation*}
     \lim_{M\to 0}\frac{I_{\alpha_j}(M)}{M^{\alpha_1}}\frac{m_j(\theta)}{m_1(\theta)}\frac{m_j(\eta)}{m_1(\eta)}=
\left\{\begin{array}{lcl}
\displaystyle\frac{1}{2^{\alpha_1}\Gamma(\alpha_1+1)} &\text{if}&j=1,\\
0 &\text{if}& j\geq 2,
\end{array}\right.
\end{equation*}
uniformly in $(\theta, \eta)\in \Theta\times\Theta$ (see \eqref{eq:Bessel_equivalent_0} and Lemma \ref{lemma_BaSm97} below), we reach the conclusion that
\begin{equation*}
     \lim_{M\to 0}\sum_{j=1}^{\infty} \frac{I_{\alpha_j}(M)}{M^{\alpha_1}}\frac{m_j(\theta)}{m_1(\theta)}\frac{m_j(\eta)}{m_1(\eta)}=\frac{1}{2^{\alpha_1}\Gamma(\alpha_1+1)},
\end{equation*}     
where the convergence is uniform for $(\theta,\eta)\in \Theta\times \Theta$. The proof of Lemma \ref{equibrotue} is complete.
\end{proof}

The following facts in the lemma below, concerning the eigenfunctions \eqref{eq:eigenfunctions}, are proved in~\cite{BaSm97}.
\begin{lemma}[\cite{BaSm97}] 
\label{lemma_BaSm97} If $C$ is normal, then
there exist two constants $0<c_1<c_2$ such that 
\begin{equation}
\label{croissalp}
     c_1j^{{1}/({d-1})}\leq \alpha_j \leq c_2 j^{{1}/{(d-1)}},\quad\forall j\geq 1.
\end{equation}
In addition, there exists a constant $c$ such that
\begin{equation}
\label{compfoncprop}
     m_j^2(\eta)\leq \frac{cm_1^2(\eta)}{I_{\alpha_j}(1)},\quad \forall j\geq 1,\quad \forall \eta\in \Theta.
\end{equation}
\end{lemma}

We conclude this section with a useful characterization of the interior of the polar cone, which was used in the proof of Theorem \ref{thm:case:A}:

\begin{lemma}
\label{polarconeinterior}
The drift vector $a$ belongs to $\open{(C^{\sharp})}$ if and only if there exists $\delta>0$ such that $\sclr{a}{y}\leq -\delta\vert y\vert$ for all $y\in \close{C}$.
\end{lemma}
\begin{proof} 
Assume first that $a$ satisfies the above condition. For all $x$ such that $\vert a-x\vert<\delta$ and all $y\in C$, we have by Cauchy-Schwarz inequality
\begin{equation*}
     \sclr{x}{y}=\sclr{a}{y}+\sclr{x-a}{y}< -\delta\vert y\vert + \delta \vert y\vert=0,
\end{equation*}     
hence $C^{\sharp}$ contains the open ball $B(a,\delta)$, and $a$ is an interior point. Conversely, suppose that there exists $r>0$ such that the closed ball $\close{B(a,r)}$ is included in $C^{\sharp}$. It is easily seen that 
\begin{equation*}
     C^{\sharp}=\{x\in\mathbb R^d :  \sclr{x}{u}\leq 0, \forall x\in\close{C}\cap \SS^{d-1}\}.
\end{equation*}
Since $\close{C}\cap\SS^{d-1}$ is a compact set, there exists a vector $u_0$ in this set such that 
\begin{equation*}
     \gamma=\sclr{a}{u_0}=\max_{u\in\close{C}\cap\SS^{d-1}}\sclr{a}{u}.
\end{equation*}      
Hence it remains to prove that $\gamma<0$. To that aim, we select a family $\{x_1,\ldots,x_d\}$ of vectors of $\partial B(a,r)$ which forms a basis of $\RR^d$. One of them, say $x_1$, must satisfy $\sclr{x_1}{u_0}<0$, since else we would have $\sclr{x_i}{u_0}=0$ for all $i\in\{1,\ldots,d\}$, and therefore $u_0=0$. Let $\close{x}_1=2a-x_1$ be the opposite of $x_1$ on $\partial B(a,r)$. Since $\sclr{x_1}{u_0}<0$ and $\sclr{\close{x}_1}{u_0}\leq 0$, it follows that $\gamma=\sclr{a}{u_0}=(\sclr{x_1}{u_0}+\sclr{\close{x}_1}{u_0})/2<0$.
\end{proof}

\subsection{Case \ref{case:B} (zero drift)}

The case of a driftless Brownian motion, that we consider in the present section, has already been investigated by many authors, see \cite{Sp58,DB87,DB88,BaSm97}. Define (with $\alpha_1$ as in \eqref{eq:alpha_j} and $u(y)$ as in \eqref{eq:harmonic_Brownian})
\begin{equation*}
     \kappa_B=\frac{1}{2^{\alpha_1}\Gamma(\alpha_1+1)}\int_C u(y)e^{-\vert y\vert^2/2} \text{d}y. 
\end{equation*}

\begin{theorem}
\label{thm:case:B}
Let $C$ be a normal cone. If the drift $a$ is zero, then
\begin{equation*}
     \PP_x[\tau_C>t]= \kappa_B u(x) t^{-p_1/2}(1+o(1)),\quad t\to\infty.
\end{equation*}     
\end{theorem}
Although a proof of Theorem \ref{thm:case:B} can be found in \cite{DB87,BaSm97}, for the sake of completeness we wish to write down some of the details below. As we shall see, the arguments are very similar to those used for proving Theorem \ref{thm:case:A}. 
\begin{proof}[Proof of Theorem \ref{thm:case:B}]
We have $a=0$ and $\gamma=0$. Thus, the lower and upper bounds \eqref{eq:bound_for_I_delta_case_A} for $I_\delta(t)$ write
\begin{equation*}
(1\pm\epsilon) bu(x)e^{-\vert x\vert^2/2}t^{d/2}\int_{\{y\in C: \vert y\vert\leq \delta\}} u(y) e^{-t\vert y\vert^2/2}\text{d}y.
\end{equation*}
This time, we make the change of variables $v=\sqrt{t}y$ and use the homogeneity of $u$ in order to transform this expression into
\begin{equation*}
(1\pm\epsilon) bu(x)e^{-\vert x\vert^2/2}e^{-t\vert a\vert^2/2}t^{-p_1/2}\int_{\{v\in C: \vert v\vert\leq \sqrt{t}\delta\}} u(v) e^{-\vert v\vert^2/2}\text{d}v.
\end{equation*}
Since the function $u(v)e^{-\vert v\vert^2/2}$ is integrable on $C$, it comes from the dominated convergence theorem that
\begin{equation*}
\int_{\{v\in C: \vert v\vert\leq \sqrt{t}\delta\}} u(v) e^{-\vert v\vert^2/2}\text{d}v=(1+o_\delta(1))\int_{C} u(v) e^{-\vert v\vert^2/2}\text{d}v, \quad t\to\infty.
\end{equation*}
Hence, the bounds for $I_\delta(t)$ can finally be written as
\begin{equation*}
(1\pm\epsilon) \kappa_B u(x)e^{-\vert x\vert^2/2}t^{-p_1/2}(1+o_\delta(1)), \quad t\to\infty.
\end{equation*}
The theorem then follows by an application of Lemma~\ref{lem:general_proof_strategy} and formula~\eqref{exittime_after_change}.
\end{proof}

\subsection{Case \ref{case:C} (interior drift)} 
Now we turn to the case when the drift $a$ is inside the cone $C$.

\begin{theorem}
\label{thm:case:C}
Let $C$ be a normal cone. If $a$ belongs to $C$, then
\begin{equation*}
     \PP_x[\tau_C= \infty]=\lim_{t\to\infty}\PP[\tau_C > t]=(2\pi)^{d/2}e^{\vert x-a\vert^2}p^{C}(1,x,a).
\end{equation*}     
\end{theorem}

\begin{proof} 
Since $a\in C$, one has $p_C(a)=a$ and $\gamma=0$. As in the previous cases, we focus our attention on
\begin{equation*}
\label{eq:starting_formula_case_C}
     I_{\delta}(t)=t^{d/2}\int_{\{y\in C:\vert y-a\vert\leq\delta \}} e^{\vert y\vert^2/2}p^C(1,x,y) e^{-t\vert a-y\vert^2/2}\text{d}y.
\end{equation*}
Given $\epsilon>0$, we choose  $\delta>0$ so small that
$\close{B(a,\delta)}\subset C$ and 
\begin{equation*} f(y)=e^{\vert y\vert^2/2}p^C(1,x,y)
\end{equation*} 
be bounded from above and below by
$f(a)\pm\epsilon$ for all $y\in \close{B(a,\delta)}$. With this choice, $I_{\delta}(t)$ is then bounded from above and below by
\begin{equation*}
     (f(a)\pm \epsilon) t^{d/2}\int_{\{y\in \RR^d : \vert y-a\vert\leq \delta\}}e^{-t\vert y-a\vert^2/2}\text{d}y.
\end{equation*}
By the change of variables $v=\sqrt{t}(y-a)$, this expression becomes
\begin{equation*}
 (f(a)\pm \epsilon)\int_{\{v\in\RR^d : \vert v\vert\leq  \sqrt{t}\delta\}} e^{-\vert v\vert^2/2}\text{d}v
 =(f(a)\pm \epsilon)(2\pi)^{d/2}(1+o_{\delta}(1)),\quad t\to\infty.
\end{equation*} 
Hence, the theorem follows from Lemma \ref{lem:general_proof_strategy} and formula \eqref{exittime_after_change}.
\end{proof}

\begin{example}
In the case where $C$ is the Weyl chamber of type $A$,
the heat kernel is given by the Karlin-McGregor formula (see \cite[Theorem 1]{KaMcGr59}):
\begin{equation*}
     p^C(t,x,y)=\det(p(t,x_i,y_j))_{1\leq i,j\leq d},
\end{equation*}     
with $p(t,x_i,y_j)=(2\pi t)^{-1/2}e^{-(x_i-y_j)^2/2t}$. An easy computation then shows that $p^C(1,x,a)$ is equal to
\begin{equation*}
     p^C(1,x,a)=(2\pi)^{-d/2}e^{-(\vert x\vert^2+\vert a\vert^2)/2}\det(e^{x_ia_j})_{1\leq i,j\leq d}.
\end{equation*}     
Hence 
\begin{equation*}
     \PP_x[\tau_C= \infty]=\lim_{t\to\infty}\PP_x[\tau_C>t]=e^{\sclr{-a}{x}}\det(e^{x_ia_j})_{1\leq i,j\leq d}.
\end{equation*}     
\end{example}

This result was derived earlier by Biane, Bougerol and O'Connell in \cite[section 5]{BiBoOC05}. Indeed, in \cite{BiBoOC05} the authors first find the probability $\PP_x[\tau_C= \infty]$ in the case of a drift $a\in C$ via the reflection principle and a change of measure. As an application of this, they show that the Doob $h$-transform of the Brownian motion with the harmonic function given by the non-exit probability $\PP_x[\tau_C= \infty]$ has the same law that a certain path transformation of the Brownian motion (defined thanks to the Pitman operator, which is one of the main topics studied in \cite{BiBoOC05}).

\subsection{Case \ref{case:D} (boundary drift)}
\label{sec:case:D}

In this section and the following ones, we make the additional hypothesis that the cone is real-analytic, that is, hypothesis \ref{hypothesis2}.
Notice that under this assumption $\Theta$ is normal.
This assumption ensures that the heat kernel can be locally and analytically continued across the boundary, and thus admits a Taylor expansion at any boundary point different from the origin. 
To our knowledge, for more general cones like those which are intersections of smooth deformations of half-spaces, the boundary behavior of the heat kernel at a corner point (i.e., a point located at the intersection of two or more half-spaces) is not known, except in the particular case of Weyl chambers \cite{KaMcGr59,BiBoOC05}.
This behavior will determine the polynomial part $t^{-\alpha}$ in the {asymptotics} of the non-exit probability. The case of Weyl chambers is treated in \cite{PuRo08}. 
Here, we deal with the opposite (i.e., smooth) setting.

Define the function
\begin{equation*}
     h_D(x) = e^{\vert x-a\vert^ 2/2}\partial_n p^C(1,x,a),
\end{equation*}
where $n$ stands for the inner-pointing unit vector normal to $C$ at $a$, and $\partial_n p^C(1,x,a)$ denotes the normal derivative of the function $y\mapsto p^C(1,x,y)$ at $y=a$. Function $h_D(x)$ is non-zero thanks to Lemma \ref{lemma:normal_derivative_at_a} below. Define also the constant
\begin{equation*}
       \kappa_D = ({2\pi})^{(d-1)/2}.
\end{equation*}
\begin{theorem}
\label{thm:case:D}
Let $C$ be a real-analytic cone. If $a\not=0$ belongs to $\partial C$, then 
\begin{equation*}
     \PP_x[\tau_C>t]=\kappa_D h_D(x)t^{-1/2}(1+o(1)),\quad t\to\infty.
\end{equation*}
\end{theorem}

\begin{proof}
As in case \ref{case:C}, we have $p_C(a)=a$ and $\gamma=0$, and the formula \eqref{eq:defintion_of_I_delta} for $I_\delta(t)$ writes
\begin{equation*}
\label{eq:starting_formula_case_D}
     I_{\delta}(t)=t^{d/2}\int_{\{y\in C:\vert y-a\vert\leq\delta \}} f(y) e^{-t\vert a-y\vert^2/2}\text{d}y,
\end{equation*}
where $f(y)=e^{\vert y\vert^2/2}p^C(1,x,y)$. In the present case, $f(y)$ vanishes at $y=a$, contrary to case \ref{case:C}. Since the function $y\mapsto p^C(1,x,y)$ is infinitely differentiable in a neighborhood of $a$ (see Lemma \ref{lemma:heat_kernel_continuation}), it follows from Taylor's formula 
 that, for any (sufficiently small) $\delta>0$, there exists $M>0$ such that
\begin{equation*}
\label{eq:application_Taylor}
     \vert f(y) -\sclr{y-a}{\nabla  f(a)}\vert \leq M\vert y-a\vert ^2,\quad \forall\vert y-a\vert\leq \delta.
\end{equation*}
Therefore, for any fixed $\delta>0$, one has
\begin{equation*}
  I_\delta(t)=t^{d/2}\int_{\{y\in C : \vert y-a\vert\leq \delta\}}  (\sclr{y-a}{\nabla  f(a)}+O(\vert y-a\vert^2))e^{-t\vert y-a\vert^2/2}\text{d}y.
\end{equation*}
Making the change of variables $v=\sqrt{t}(y-a)$ implies that the above equation is the same as
\begin{equation*}
     {t^{-1/2}} \int_{(C-\sqrt{t}a)\cap \{v\in \RR^d : \vert v\vert \leq \sqrt{t}\delta\}}\sclr{v}{\nabla  f(a)}e^{-\vert v\vert^2/2}\text{d}v+O(t^{-1}).
\end{equation*}
Now, due to the regularity of $\partial C$ at $a$, the set
\begin{equation*}
     (C-\sqrt{t}a)\cap \{v\in \RR^d : \vert v\vert \leq \sqrt{t}\delta\}
\end{equation*}
goes to $\{v\in \RR^d : \sclr{v}{n}>0\}$ as $t\to\infty$. Furthermore, an easy computation shows that
\begin{equation*}
     \int_{\{v\in \RR^d : \sclr{v}{n}>0\}} v e^{-\vert v\vert^2/2}\text{d}v=({2\pi})^{(d-1)/2}n.
\end{equation*} 
Hence, we deduce that 
\begin{equation*}
     I_{\delta}(t)=t^{-1/2} ({2\pi})^{(d-1)/2}\partial_n f(a)+o_{\delta}(t^{-1/2}),\quad t\to\infty.
\end{equation*}     
Since $\partial_n f(a)=e^{\vert a\vert^2/2}\partial_n p^C(1,x,a)\not=0$
by Lemma \ref{lemma:normal_derivative_at_a}, Theorem \ref{thm:case:D} follows from Lemma \ref{lem:general_proof_strategy} and formula \eqref{exittime_after_change}.
\end{proof}

The two following lemmas have been used in the proof of Theorem \ref{thm:case:D}. The first of the two lemmas follows from \cite[Theorem 1]{KiNi78}, which proves analyticity of solutions to general parabolic problems, both in the interior and on the boundary. As for the second one, it is a consequence of \cite[Theorem 2]{Fr58}.

\begin{lemma}
\label{lemma:heat_kernel_continuation}
Under \ref{hypothesis2}, the function $y\mapsto p^C(1,x,y)$ can be analytically continued in some open neighborhood of any point $y=a\in\close{C}\setminus\{0\}$.
\end{lemma}

\begin{lemma}
\label{lemma:normal_derivative_at_a}
Under \ref{hypothesis2}, the normal derivative of the function $y\mapsto p^C(1,x,y)$ at $y=a$ is non-zero {\rm(}for any $a\in \partial C\setminus \{0\}${\rm)}.
\end{lemma}

\setcounter{example}{0}

\begin{example}[continued]
In the particular case of the dimension $2$, with a cone of opening angle $\beta$ (see Figure \ref{fig:cones}), one has the following expression for the normal derivative at $a$:
\begin{equation*}
     \partial_n f(a) = \frac{2\pi}{\vert a\vert \beta^2}e^{-\vert x\vert^2/2}\sum_{j=1}^{\infty} I_{\alpha_j}(\vert x\vert \vert a\vert)m_j(x)j,
\end{equation*}
which gives a simplified expression for function $h_D(x)$. The above identity is elementary: it follows from the expression \eqref{eq:expression_eigenfunctions_2} of the eigenfunctions together with the definition of function $f$ and some uniform estimates (to be able to exchange the summation and the derivation in the series defining the heat kernel). 
\end{example}

\subsection{Case \ref{case:E} (non-polar exterior drift)}
\label{sec:case:E}
In addition to the real-analyticity \ref{hypothesis2} of the cone, we shall assume in this section that \ref{hypothesis3} holds, i.e., that the set $\Pi(a)=\{y\in \close{C}: \vert y-a\vert=d(a,C)\}$ of minimum points is finite. Define the function
\begin{equation*}
     h_E(x) = e^{\vert x-a\vert^2/2}\sum_{p\in\Pi(a)}\kappa_E(p)\partial_n p^C(1,x,p),
\end{equation*}
where $\kappa_E(p)$ denotes the positive constant defined in equation~\eqref{eq:def_kappa_E}. We shall prove the following:

\begin{theorem}
\label{thm:case:E}
Let $C$ be a real-analytic cone. If $a$ belongs to $\RR^d\setminus (\overline{C}\cup C^\sharp)$ and $\Pi(a)$ is finite, then 
\begin{equation*}
     \PP_x[\tau_C>t]=h_E(x)t^{-3/2}e^{-td(a,C)^2/2}(1+o(1)),\quad t\to\infty.
\end{equation*}     
\end{theorem}

\begin{proof}
Since $a$ belongs to $\RR^d\setminus (\overline{C}\cup C^\sharp)$, every $p\in\Pi(a)$ belongs to $\partial C$ and is different from $0$ and $a$.
Here $\gamma=\vert p-a\vert^2/2$. 

Because $\Pi(a)$ is finite, we can choose $\delta>0$ so that the balls $B(p,\delta)$ for $p\in\Pi(a)$ are pairwise disjoint. Then $I_{\delta}(t)$ can be written as
\begin{equation*}
     I_{\delta}(t)=\sum_{p\in\Pi(a)} I_{\delta, p}(t),
\end{equation*}        
where
\begin{equation*}
     I_{\delta,p}(t)=t^{d/2}\int_{C\cap B(p,\delta)}e^{\vert y\vert^2/2}p^C(1,x,y)e^{-t\vert a-y\vert^2/2} \text{d}y,
\end{equation*}        
and $B(p,\delta)$ does not contain any other element of $\Pi(a)$ than $p$.

The beginning of the analysis of $I_{\delta,p}(t)$ is similar to the proof of Theorem \ref{thm:case:D}, except that we have to make a Taylor expansion with three (and not two) terms, for reasons that will be clear later.
For the same reasons as in case \ref{case:D}, for any $\delta>0$ small enough, we have
\begin{multline}
\label{eq:bcov1}
I_{\delta,p}(t)=t^{d/2}\int \left(\sclr{y-p}{\nabla f (p)}+\frac{1}{2}(y-p)^{\top}\nabla ^2 f(p)(y-p)+O(\vert y-p\vert^3)\right)\\
\times e^{-t\vert y-a\vert^2/2}\text{d}y,
\end{multline}
where $f(y)=e^{\vert y\vert^2/2}p^C(1,x,y)$, $(y-p)^{\top}$ is the transpose of the vector $y-p$, $\nabla ^2 f(p )$ denotes the Hessian matrix of $f$ at $p$, and the domain of integration is $\{y\in C : \vert y-p\vert<\delta\}$. To compute the {asymptotics} of the integral $I_{\delta,p}(t)$ as $t\to\infty$, we shall make a series of two changes of variables. First, the change of variables $u=y-p$ and the use of the identity
\begin{equation*}
     e^{-t\vert y-a\vert^2/2} = e^{-t\gamma}e^{-t\vert y-p\vert^2/2-t\sclr{y-p}{p-a}}
\end{equation*}
give the following alternative expression 
\begin{equation}
\label{eq:bcov2}
  I_{\delta,p}(t)=t^{d/2}e^{-t\gamma}\int_D  \left(\sclr{u}{\nabla  f (p)}+\frac{1}{2}u^{\top}\nabla ^2 f(p )u+O(\vert u\vert^3)\right) e^{-t\vert u\vert^2/2}e^{-t\sclr{u}{p-a}}\text{d}u,
\end{equation}
where the domain of integration $D$ equals $(C-p)\cap \{u\in \mathbb R^d : \vert u\vert<\delta\}$.

In what follows, we will assume (without loss of generality) that the inner-pointing unit normal to $\partial C$ at $p$ is equal to $e_1$, the first vector of the standard basis. With this convention $p-a=\vert p-a\vert e_1$, and the only non-zero component of $\nabla  f (p)$ is in the $e_1$-direction. Indeed, since $f(y)=0$ for $y\in\partial C$, the boundary of the cone is a level set for the function $f$, and it is well known that the gradient is orthogonal to the level curves. Therefore, the quantity $\sclr{u}{\nabla  f (p)}$ is equal to $u_1\partial_1 f(p)$.

Our last change is $v=\phi_t(u)$; it sends $(u_1,u_2,\ldots ,u_d)$ onto $(tu_1,\sqrt{t} u_2,\ldots ,\sqrt{t}u_d)$. Note that the scalings in the normal and tangential directions are not the same; this entails that in \eqref{eq:bcov1} the second term in the integrand is not negligible w.r.t.\ the first one, and this is the reason why we have to make a Taylor expansion with three terms and not two. 
Note also that the Jacobian of this transformation is $t^{(d+1)/2}$. From this and \eqref{eq:bcov2} we deduce that, as $t\to\infty$,
\begin{multline}
\label{eq:bcov3}
t^{3/2}e^{t\gamma}I_{\delta,p}(t)=\int_{\phi_t(D)} \left(v_1 \partial_1 f(p)
+\frac{1}{2} (0,v_2,\ldots ,v_d)^\top \nabla ^2 f(p )(0,v_2,\ldots ,v_d)\right)\\
\times e^{- v_1 \vert p-a\vert}e^{-(v_2^2+\cdots +v_d^2)/2}e^{-v_1^2/(2t)}\text{d}v
+O(t^{-1/2}).
\end{multline}
The aim is now to understand the behavior of the domain $\phi_t(D)$ as $t\to\infty$. 
Since the cone $C$ is tangent to the hyperplane $\{u\in\RR^d : u_1=0\}$ at $p$ and its boundary is real-analytic, there exists a real-analytic function 
$g$ with $g(0)=0$ and $\nabla g(0)=0$, such that, for $\delta$ small enough, the domain $D$ coincides with 
\begin{equation*}
     \{u\in\mathbb R^d : u_1 > g(u_2,\ldots ,u_d), \vert u\vert < \delta\}.
\end{equation*}
An application of Taylor formula then gives that (up to a set of Lebesgue measure zero)
\begin{equation*}
     \lim_{t\to\infty} \phi_t(D) = \phi_\infty(D)=\{v\in\mathbb R^d : v_1 > \frac{1}{2}(v_2,\ldots ,v_d)^\top\nabla^2g(0)(v_2,\ldots ,v_d)\}.
\end{equation*}
Let us compare the limit domain $\phi_\infty(D)$ and the integrand in equation \eqref{eq:bcov3}. Since $f$ vanishes on the boundary of the cone, we have 
\begin{equation*}
\label{eq:id_zero}
     f(p_1+g(u_2,\ldots ,u_d),p_2+u_2,\ldots ,p_d+u_d)=0,
\end{equation*}
for any $u$ in some neighborhood of $0$.
Differentiating twice this identity, we obtain
\begin{equation*}
    (0,v_2,\ldots ,v_d)^\top\nabla^2f(p   )(0,v_2,\ldots ,v_d)=-\partial_1f(p) (v_2,\ldots ,v_d)^\top\nabla^2g(0)(v_2,\ldots ,v_d).
\end{equation*}
Therefore, equation \eqref{eq:bcov3} can be rewritten as
\begin{multline*}
t^{3/2}e^{t\gamma}I_{\delta,p}(t)=\partial_1f(p)\int_{\phi_t(D)} \left(v_1 -\frac{1}{2} (v_2,\ldots ,v_d)^\top \nabla ^2 g(0)(v_2,\ldots ,v_d)\right)\\
\times e^{- v_1 \vert p-a\vert}e^{-(v_2^2+\cdots +v_d^2)/2}e^{-v_1^2/(2t)}\text{d}v
+O(t^{-1/2})
\end{multline*}
as $t\to\infty$. Notice that the limit domain $\phi_\infty(D)$ is exactly the subset of $\RR^d$ where the integrand is positive. Thus, the constant
\begin{multline}
\label{eq:def_kappa_E}
     \kappa_E(p)=e^{(\vert p\vert^2-\vert a\vert^2)/2} \\\times\int_{\phi_\infty(D)} \left(v_1-(v_2,\ldots ,v_d)^\top\nabla^2g(0)(v_2,\ldots ,v_d)\right)e^{- v_1 \vert p-a\vert}e^{-(v_2^2+\cdots +v_d^2)/2}\text{d}v
\end{multline}
is positive.
Since $\partial_1f(p)=e^{\vert p\vert^2/2}\partial_1 p^C(1,x,p)\not=0$ by Lemma \ref{lemma:normal_derivative_at_a}, we obtain that
\begin{equation*}
I_{\delta,p}(t)=\kappa_E(p) e^{\vert a\vert^2/2}\partial_1p^C(1,x,p)t^{-3/2}e^{-t\gamma}(1+o(1)),\quad t\to\infty.
\end{equation*}
To conclude the proof of Theorem \ref{thm:case:E}, it suffices to sum the estimates for $I_{\delta, p}(t)$ over $p\in\Pi(a)$, and then to apply Lemma \ref{lem:general_proof_strategy} and to use equation \eqref{exittime_after_change}.
\end{proof}

\setcounter{example}{0}

\begin{example}[continued]
In the particular case of two-dimensional cones, $\nabla^2 g(0)=0$ and the limit domain of integration $\phi_\infty(D)$ is the half-space $\{v\in\mathbb R^2 : v_1\geq 0\}$. The constant $\kappa_E(p)$ can then be computed:
\begin{equation*}
     \kappa_E(p) = \frac{e^{(\vert p\vert^2-\vert a\vert^2)/2}}{\vert p-a\vert^2}\sqrt{2\pi}.
\end{equation*}
\end{example}

\subsection{Case \ref{case:F} (polar boundary drift)}
\label{subsec:case:F}

We finally consider the case where the drift $a\not=0$ belongs to $\partial C^\sharp$. Let us first notice that the existence of such a vector $a$ implies that the cone $C$ is included in some half-space. More precisely, by definition of the polar cone, the inner product of $a$ with any $y\in C$ is non-positive, so that $C$ is included in the half-space $\{y\in\mathbb R^d: \sclr{a}{y}\leq 0\}$. Moreover, there must exist some  $\theta_c\in\partial\Theta=\partial (C\cap\SS^{d-1})$ such that $\sclr{a}{\theta_c}=0$, for else $a$ would belong to the interior of $C^{\sharp}$, as seen in Lemma \ref{polarconeinterior}. We call $\Theta_c$ the set of all these {\em contact points} $\theta_c$ between $\partial\Theta$ and the hyperplane $a^\perp=\{y\in\mathbb R^d:\sclr{a}{y}=0\}$. As we shall see, the asymptotics of $\PP_x[\tau_C>t]$ is determined by the local geometry of the cone $C$ near these points.

We first present some general aspects of our approach, and then we will treat the case $d=2$ for cones with opening angle $\beta\in (0,\pi)$, and the case $d=3$ for cones with a real-analytic boundary and a finite number of contact points. Other cases are left as open problems. In the sequel, we will assume (without loss of generality)\ that $a=-\vert a\vert e_d$, where $e_d$ stands for the last vector of the standard basis.

As in case \ref{case:A}, we have $p_C(a)=0$ and $\gamma=\vert a\vert^2/2$, so that the formula \eqref{eq:defintion_of_I_delta} for $I_\delta(t)$ can be written as
\begin{equation*}
\label{eq:starting_formula_case_F2}
     I_{\delta}(t)=t^{d/2}\int_{\{y\in C:\vert y\vert\leq\delta \}} e^{\vert y\vert^2/2}p^C(1,x,y) e^{-t\vert a-y\vert^2/2}\text{d}y.
\end{equation*}
Let $\epsilon>0$ be given. Arguing as in case \ref{case:A}, we can pick $\delta>0$ small enough so that $I_\delta(t)$ be bounded from above and below by
\begin{equation}
\label{eq:bound_for_I_delta_case_F2}
(1\pm\epsilon) bu(x)e^{-\vert x\vert^2/2}t^{d/2}\int_{\{y\in C: \vert y\vert\leq \delta\}} u(y) e^{-t\vert a-y\vert^2/2}\text{d}y,
\end{equation}
where $b=(2^{\alpha_1}\Gamma(\alpha_1+1))^{-1}$. Thus, we are led to study the asymptotic behavior of
\begin{align*}
\label{eq:definition_of_J_delta_case_F2}
J_{\delta}(t)&=t^{d/2}\int_{\{y\in C: \vert y\vert\leq \delta\}} u(y) e^{-t\vert a-y\vert^2/2}\text{d}y,\\
             &=e^{-t\gamma}t^{d/2}\int_{\{y\in C: \vert y\vert\leq \delta\}} u(y) e^{-t\vert y\vert^2/2}e^{-t\vert a\vert y_d}\text{d}y.
\end{align*}
Making the change of variables $z=\sqrt{t}y$ and using the homogeneity property of $u$ (see \eqref{eq:harmonic_Brownian}), we obtain
\begin{equation}
\label{eq:J_delta_after_first_change_case_F2}
 J_{\delta}(t)=e^{-t\gamma}t^{-p_1/2}\int_{\{z\in C: \vert z\vert\leq\sqrt{t} \delta\}} u(z) e^{-\vert z\vert^2/2}e^{-\sqrt{t}\vert a\vert z_d}\text{d}z.
\end{equation}
Now, Laplace's method suggests that only some neighborhood of the hyperplane $\{z\in\mathbb R^d :z_d=0\}$ will contribute to the asymptotics. More precisely, we have the following result:

\begin{lemma}
\label{lem:spherical_cap_contribution}
For any $\eta>0$, we have
\begin{equation*}
     \int_{\{z\in C: z_d>\eta\vert z\vert\}} u(z) e^{-\vert z\vert^2/2}e^{-\sqrt{t}\vert a\vert z_d}\textnormal{d}z=o(t^{-d/2}),\quad t\to\infty.
\end{equation*}     
\end{lemma}
\begin{proof} 
Since $\vert u(z)\vert\leq M \vert z\vert^{p_1}$, the integral above is bounded from above by
\begin{equation*}
     M\int_{\RR^d}\vert  z\vert^{p_1}e^{-\eta\sqrt{t}\vert a\vert\vert z\vert} \text{d}z=Mt^{-(p_1+d)/2}\int_{\RR^d}\vert  w\vert^{p_1}e^{-\eta\vert a\vert\vert w\vert} \text{d}w,
\end{equation*}     
which is equal to $O(t^{-(p_1+d)/2})$. Lemma \ref{lem:spherical_cap_contribution} follows since $p_1>0$.
\end{proof}

From now on, we shall assume that \ref{hypothesis4} holds, i.e., that the set of contact points $\Theta_c$ is finite.

Let $\eta>0$ be so small that the $d$-dimensional balls $B(\theta_c, \eta)$ for $\theta_c\in\Theta_c$ are disjoints. Since the set of all $\theta\in\close{\Theta}$ that do not belong to any of these open balls is compact and does not contain any contact point, there exists some $\eta'>0$ such that
$\theta_d>\eta'$ for all such $\theta$.
For $\theta_c\in\Theta_c$, we define the cone
\begin{equation}
\label{eq:thin_cones_definition}
C(\theta_c,\eta)=\{z\in C: z/\vert z\vert\in B(\theta_c,\eta)\}.
\end{equation}
Then $C$ can be written as the disjoint union of these (thin) cones and of a (big) remaining cone whose points $z$ all satisfy the inequality $z_d/\vert z\vert>\eta'$.
Thus, according to formula \eqref{eq:J_delta_after_first_change_case_F2} and Lemma \ref{lem:spherical_cap_contribution}, we have
\begin{equation}
\label{eq:decomposion_of_J_delta}
J_{\delta}(t)=e^{-t\gamma}t^{-p_1/2}\left(\sum_{\theta_c\in\Theta_c}K_{\delta, \eta,\theta_c}(t)+o(t^{-d/2})\right),
\end{equation}
where
\begin{equation}
\label{eq:contact_contribution_definition}
K_{\delta, \eta,\theta_c}(t)=\int_{\{z\in C(\theta_c,\eta): \vert z\vert\leq\sqrt{t} \delta\}} u(z) e^{-\vert z\vert^2/2}e^{-\sqrt{t}\vert a\vert z_d}\text{d}z
\end{equation} 
represents the contribution of the contact point $\theta_c$.

\subsection*{Two-dimensional cones}
Here the cone is $C=\{\rho e^{i\theta}: \rho>0, \theta\in(0,\beta)\}$ with $\beta\in(0,\pi)$.
Define 
\begin{equation*}
     h_F(x) = e^{\sclr{-a}{x}}u(x)
\end{equation*}
and the constant
\begin{equation*}
       \kappa_F =\frac{\pi 2^{p_1/2}\Gamma(p_1/2)}{2^{\alpha_1}\Gamma(\alpha_1+1)\beta^2\vert a\vert^2}.
\end{equation*}

\begin{theorem}[Case of the dimension $2$]
\label{thm:case:F}
Let $C$ be any two-dimensional cone with $\beta\in(0,\pi)$. If $a\not=0$ belongs to $\partial C^\sharp$, then 
\begin{equation*}
     \PP_x[\tau_C>t]=\kappa_F h_F(x)e^{-t\vert a\vert^2/2}t^{-(p_1/2+1)}(1+o(1)),\quad t\to\infty.
\end{equation*}   
\end{theorem}

\begin{proof}
Since $\beta<\pi$, there is only one contact point, namely $\theta_c=(1,0)$. Let us analyze its contribution.
According to \eqref{eq:contact_contribution_definition}, we have
\begin{equation*}
K_{\delta, \eta,\theta_c}(t)=\int_{\{z\in\mathbb R^2: 0<z_2<\eta z_1,\vert z\vert\leq\sqrt{t} \delta\}} u(z) e^{-\vert z\vert^2/2}e^{-\sqrt{t}\vert a\vert z_2}\text{d}z,
\end{equation*}
as soon as $\eta$ is small enough. (In fact, the condition is $\arcsin \eta <\beta$, and $\eta$ in the integral should be $\tan(\arcsin\eta)$.)

We now proceed to the change of variables $v=\phi_t (z)=(z_1,\sqrt{t}z_2)$, which leads to 
\begin{equation*}
     K_{\delta, \eta,\theta_c}(t)=t^{-1/2}\int_{D_t} u\left(v_1,\frac{v_2}{\sqrt{t}}\right) e^{-\vert v_1\vert^2/2}e^ {-\vert v_2\vert/2t}e^{-\vert a\vert v_2}\text{d}v,
\end{equation*}     
where $D_t=\phi_t(\{z\in\mathbb R^2:0<z_2<\eta z_1,\vert z\vert\leq\sqrt{t} \delta\})$. Notice that $(v_1,v_2)\in D_t$ implies that $\vert v_2/(v_1\sqrt{t})\vert<\eta$.
It follows from the Taylor-Lagrange inequality that (if $\eta$ is small enough) there exists $M$ such that
\begin{equation*}
     u(1,h)=\partial_2 u(1,0) h+h^2 R(h), 
\end{equation*}     
with $\vert R(h)\vert\leq M$ for all $\vert h\vert\leq \eta$. Therefore, using the homogeneity of $u$, we obtain
\begin{equation*}
     \sqrt{t}u\left(v_1,\frac{v_2}{\sqrt{t}}\right)=\sqrt{t}v_1^{p_1}u\left(1,\frac{v_2}{v_1\sqrt{t}}\right)=v_1^{p_1-1}v_2(\partial_2 u(1,0)+h R(h)),
\end{equation*}     
with $h=v_2/(v_1\sqrt{t})$ and $\vert hR(h)\vert\leq\eta M$ for all $(v_1,v_2)\in D_t$. As $t\to\infty$, the domain $D_t$ converges to the quarter plane $\mathbb R_+^2$, and it follows from the dominated convergence theorem that, as $t\to\infty$,
\begin{align}
\label{eq:final_expression_for_contribution_of_10}
     K_{\delta, \eta,\theta_c}(t)&=t^{-1}\partial_2u(1,0)\int_{\RR_+^2} v_1^{p_1-1}v_2 e^{-v_1^2/2}e^{-\vert a\vert v_2}\text{d}v+o(t^{-1})\\
		 &=t^{-1}\frac{\pi 2^{p_1/2}\Gamma(p_1/2)}{\beta^2\vert a\vert^2}(1+o(1)),\nonumber
\end{align}
where we have used the fact that $\partial_2 u(1,0)=2\pi/\beta^2$ (see \eqref{eq:expression_eigenfunctions_2} for $j=1$).
For $\beta<\pi$, there is no other contribution and, therefore, combining equations \eqref{eq:final_expression_for_contribution_of_10}, 
\eqref{eq:decomposion_of_J_delta} and \eqref{eq:bound_for_I_delta_case_F2} shows that upper and lower bounds for $I_{\delta}(t)$ are given by
\begin{equation*}
     (1\pm \epsilon)\kappa_Fu(x)e^{-\vert x\vert^2/2}e^{-t\gamma}t^{-({p_1}/{2}+1)}(1+o(1)),\quad t\to\infty.
\end{equation*}     
Hence, as in the other cases, the result follows from Lemma \ref{lem:general_proof_strategy} and formula \eqref{exittime_after_change}.
\end{proof}

\begin{remark}
When $\beta=\pi$, the point $(-1,0)$ is a second contact point. By symmetry, its contribution is exactly the same as that of $(1,0)$.
Hence the result of Theorem~\ref{thm:case:F} is still valid if $\kappa_F$ is replaced by $2\kappa_F$.
\end{remark}

\subsection*{Three-dimensional cones with real-analytic boundary}

Recall that (by convention) $a=-\vert a\vert e_3$ and the cone $C$ is contained in the half space $\{z_3>0\}$, see Figure \ref{fig:3Dcone}. Thanks to \eqref{eq:decomposion_of_J_delta}, the asymptotic behavior of $\PP_x[\tau_C>t]$ will follow from the study of the contributions
\begin{equation*}
K_{\delta, \eta,\theta_c}(t)=\int_{\{z\in C(\theta_c,\eta): \vert z\vert\leq\sqrt{t} \delta\}} u(z) e^{-\vert z\vert^2/2}e^{-\sqrt{t}\vert a\vert z_3}\text{d}z
\end{equation*} 
of the contact points $\theta_c\in \Theta_c$ between $\partial\Theta$ and the hyperplane $a^\perp=\{z\in\mathbb R^3: z_3=0\}$. 
As we shall see, the behavior of the integral above will depend on the geometry of $\Theta$ at the point $\theta_c$.

\unitlength=0.6cm
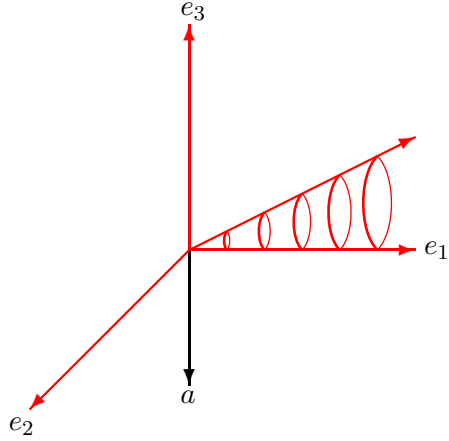
\begin{figure}[h]
\begin{center}
\begin{tabular}{ccccc}
    \begin{picture}(0,5)
    \thicklines
    \put(0,0){\textcolor{red}{\vector(0,1){5}}}
    \put(0,0){\textcolor{black}{\vector(0,-1){3}}}
    \put(0,0){\textcolor{red}{\vector(-1,-1){3.53}}}
    \put(0,0){\textcolor{red}{\vector(2,1){5}}}
    \put(0,0){\textcolor{red}{\vector(1,0){5}}}
    \psbezier[linewidth=1pt, linecolor=red](0.5,0)(0.45,0.05)(0.45,0.2)(0.5,0.25)
    \psbezier[linewidth=0.5pt, linecolor=red](0.5,0)(0.55,0.05)(0.55,0.2)(0.5,0.25)
    \psbezier[linewidth=1pt, linecolor=red](1,0)(0.9,0.1)(0.9,0.4)(1,0.5)
    \psbezier[linewidth=0.5pt, linecolor=red](1,0)(1.1,0.1)(1.1,0.4)(1,0.5)
    \psbezier[linewidth=1pt, linecolor=red](1.5,0)(1.35,0.15)(1.35,0.6)(1.5,0.75)
    \psbezier[linewidth=0.5pt, linecolor=red](1.5,0)(1.65,0.15)(1.65,0.6)(1.5,0.75)
    \psbezier[linewidth=1pt, linecolor=red](2,0)(1.8,0.2)(1.8,0.8)(2,1)
    \psbezier[linewidth=0.5pt, linecolor=red](2,0)(2.2,0.2)(2.2,0.8)(2,1)
    \psbezier[linewidth=1pt, linecolor=red](2.5,0)(2.25,0.25)(2.25,1)(2.5,1.25)
    \psbezier[linewidth=0.5pt, linecolor=red](2.5,0)(2.75,0.25)(2.75,1)(2.5,1.25)
    \put(5.2,-0.1){$e_1$}
    \put(-0.2,5.2){$e_3$}
    \put(-4,-4){$e_2$}
    \put(-0.2,-3.4){$a$}
    \end{picture}
    \end{tabular}
\end{center}
\vspace{25mm}
\caption{Three-dimensional cones in the proof of Theorem \ref{thm:case:F}}
\label{fig:3Dcone}
\end{figure}

\subsubsection*{Contribution of one fixed contact point}
Without loss of generality, let us assume that $\theta_c=e_1$. Since the cone is tangent to the plane $\{z\in\mathbb R^3: z_3=0\}$ at the point $\theta_c$ and since its boundary is assumed to be real-analytic, there exists a real-analytic function $g(z_2)$ with $g(0)=0$ and $g'(0)=0$, such that the intersection of $C$ with $\{z\in\mathbb R^3: z_1=1\}$ coincides (in a neighborhood of $\theta_c$) with the set
\begin{equation*}
\label{eq:parametric_definition_of_the_boundary_via_g}
g^+=\{z\in\mathbb R^3: z_1=1, z_3> g(z_2)\}.
\end{equation*}
Define 
\begin{equation*}
q=q(\theta_c)=\inf\{n\geq 2: g^{(n)}(0)\not=0\},
\end{equation*}
and 
\begin{equation*}
c=c(\theta_c)=\frac{g^{(q)}(0)}{q!}.
\end{equation*}
Since $\theta_c$ is isolated from the other contact points (recall that $\Theta_c$ is assumed to be finite), the function $g(z_2)$ must be positive for all $z_2\not=0$ in a neighborhood of $0$. Thus, by real-analyticity, $q$ must be finite, even, and such that $g^{(q)}(0)>0$. 
Set
\begin{equation*}
\kappa(q)=\frac{2^{(p_1+1-1/q)/{2}}(1-\frac{1}{q+1})}{\vert a\vert^{2+{1}/{q}}}\Gamma\left(\frac{p_1+1-1/q}{2}\right)\Gamma\left(2+\frac{1}{q}\right).
\end{equation*}
Then we have:
\begin{lemma}
\label{lem:individual_contact_point_contribution} 
For any $\delta>0$ and $\eta>0$ small enough, the contribution of each contact point $\theta_c$ to the asymptotics of the non-exit probability is given by
\begin{equation*}
     K_{\delta, \eta,\theta_c}(t)= \frac{\kappa(q) \partial_n u(\theta_c)}{c(\theta_c)^{1+{1}/{q}}} t^{-(1+{1}/{(2q)})} (1+o(1)),\quad t\to\infty,
\end{equation*}
where $\partial_n u(\theta_c)$ stands for the (inner-pointing) normal derivative of the function $u$ at $\theta_c$.
\end{lemma}
We postpone the proof of Lemma \ref{lem:individual_contact_point_contribution} after the statement and the proof of Theorem \ref{thm:case:F3}.

\subsubsection*{Statement of Theorem \ref{thm:case:F3}}
Let $q_1$ be the maximum value of $q(\theta_c)$ for $\theta_c\in \Theta_c$. We define 
\begin{equation*}
h_F(x)=u(x) e^{-\sclr{a}{x}}
\end{equation*}
as well as 
\begin{equation*}
\kappa_F=b\kappa(q_1)\sum_{q(\theta_c)=q_1}\frac{\partial_n u(\theta_c)}{c(\theta_c)^{1+{1}/{q}}},
\end{equation*}
where $b=(2^{\alpha_1}\Gamma(\alpha_1+1))^{-1}$.
Then we have:

\setcounter{theorem}{5}

\begin{theorem}[Case of the dimension $3$]
\label{thm:case:F3}
Let $C$ be a real-analytic three-dimensional cone. If $a\not=0$ belongs to $\partial C^\sharp$ and the set of contact points $\Theta_c$ between $\partial\Theta$ and the hyperplane $a^\perp$ is finite, then 
\begin{equation*}
     \PP_x[\tau_C>t]=\kappa_F h_F(x) t^{-({p_1}/{2}+1+{1}/({2q_1}))} e^{-t\vert a\vert^2/2}(1+o(1)),\quad t\to\infty.
\end{equation*}
\end{theorem}
\begin{proof}
Since $K_{\delta, \eta,\theta_c}(t)$ is of order $t^{-(1+{1}/{(2q)})}$ by Lemma \ref{lem:individual_contact_point_contribution}, only those $\theta_c$ with $q(\theta_c)=q_1$ will contribute in \eqref{eq:decomposion_of_J_delta} to the asymptotics of $J_{\delta}(t)$. Thus, we obtain that
\begin{equation*}
     J_{\delta}(t)=e^{-t\gamma}t^{-({p_1}/{2}+1+{1}/{(2q_1)})}\kappa(q_1)\sum_{q(\theta_c)=q_1}\frac{ \partial_n u(\theta_c)}{c(\theta_c)^{1+{1}/{q_1}}}(1+o(1)),\quad t\to\infty.
\end{equation*}     
Now, equation \eqref{eq:bound_for_I_delta_case_F2} shows that bounds for $I_{\delta}(t)$ are given by
\begin{equation*}
     (1\pm \epsilon)\kappa_Fu(x)e^{-\vert x\vert^2/2}e^{-t\gamma}t^{-({p_1}/{2}+1+{1}/{(2q_1)})}(1+o(1)),\quad t\to\infty.
\end{equation*}     
Hence, the result follows from Lemma \ref{lem:general_proof_strategy} and formula \eqref{exittime_after_change}.
\end{proof}

\begin{proof}[Proof of Lemma \ref{lem:individual_contact_point_contribution}]
With the conventions made just above, we analyze the contribution of $\theta_c=(1,0,0)$, namely,
\begin{equation*}
     K_{\delta, \eta,\theta_c}(t)=\int_{\{z\in C(\theta_c,\eta): \vert z\vert\leq\sqrt{t} \delta\}} u(z) e^{-\vert z\vert^2/2}e^{-\sqrt{t}\vert a\vert z_3}\text{d}z.
\end{equation*}
By making the linear change of variables $v=\phi_t(z)$, with
\begin{equation*}
     \phi_t(z_1,z_2,z_3)=(z_1,t^{1/(2q)}z_2,\sqrt{t}z_3),
\end{equation*}     
we obtain
\begin{equation}
\label{eq:last_equation_for_J_delta}
K_{\delta, \eta,\theta_c}(t)=t^{-{1}/{2}-{1}/({2q})}\int_{D_t} u\left(v_1,\frac{v_2}{t^{1/(2q)}},\frac{v_3}{\sqrt{t}}\right) e^{-v_1^2/2}e^{-\vert a\vert v_3}(1+o(1))\text{d}v,
\end{equation}
where $D_t=\phi_t(\{z\in C(\theta_c,\eta): \vert z\vert\leq\sqrt{t} \delta\})$, and $1+o(1)$ increases to $1$ as $t\to\infty$. 

In order to understand the behavior of $D_t$ as $t\to\infty$, we first notice that
\begin{equation*}
     \lim_{t\to\infty} D_t=\lim_{t\to\infty}\phi_t(C(\theta_c,\eta)).
\end{equation*}
Then, since the first coordinate is left invariant by $\phi_t$, we shall look at what happens in the plane $\{z_1=1\}$. It follows from the definition of $q$ that
\begin{equation*}
     g^+=\{z\in\mathbb R^3: z_1=1, z_3 > cz_2^q+o(z_2^q)\},
\end{equation*}     
with $c=g^{(q)}(0)/q!>0$. From this and the definition of $\phi_t$, it is easily seen that
\begin{equation*}
     \lim_{t\to\infty}\phi_t(C(\theta_c,\eta)\cap\{z\in\mathbb R^3: z_1=1\})=\{v\in\mathbb R^3: v_1=1, v_3> cv_2^q\}.
\end{equation*}     
Further, the homogeneity of the cone and the linearity of $\phi_t$ immediately imply that
\begin{equation*}
     \lim_{t\to\infty}\phi_t(C(\theta_c,\eta)\cap\{z\in\mathbb R^3: z_1=\lambda\})=\{v\in\mathbb R^3: v_1=\lambda, \lambda^{q-1}v_3> cv_2^q\},
\end{equation*}     
for all $\lambda>0$. Now, if $\eta>0$ is small enough, the cone $C(\theta_c,\eta)$ does not contain any $z$ with $z_1\leq 0$.
Therefore,
\begin{equation}
\label{eq:limit_domain_in_case_F}
\lim_{t\to\infty}\phi_t(C(\theta_c,\eta))=\{v\in\mathbb R^3: v_1>0, v_3>0, v_1^{q-1}v_3>cv_2^q\}.
\end{equation}
We call $D$ the limit domain in \eqref{eq:limit_domain_in_case_F}.

It remains to analyze the behavior of the integrand in \eqref{eq:last_equation_for_J_delta}, i.e., to find the asymptotics of
\begin{equation*}
     u\left(v_1,\frac{v_2}{t^{1/(2q)}}, \frac{v_3}{\sqrt{t}}\right)=v_ 1^{p_1}u\left(1,\frac{v_2}{v_1t^{1/(2q)}}, \frac{v_3}{v_1\sqrt{t}}\right)
\end{equation*}     
for $v_1>0$, as $t\to\infty$. To this end, we shall use a Taylor expansion of $u(1,x,y)$ in a neighborhood of $(0,0)$. This can be done since it is known that the real-analyticity of $\Theta$ ensures that $u$ can be extended to a strictly bigger cone, inside of which $u$ is (still) harmonic, see \cite[Theorem A]{MoNi57}. Since $u$ is equal to zero on the boundary of $C$, the relation
\begin{equation*}
     u(1,z_2,g(z_2))=0
\end{equation*}     
holds for all $z_2$ in a neighborhood of $0$, and a direct application of Lemma \ref{lem:implicit_vs_parametric_derivatives} below for $n=1$ and $k\in\{0,\ldots, q-1\}$ shows that 
\begin{equation}
\label{eq:link_between_partial_derivatives}
\partial_{2,2,\ldots,2}^{(j)}u(1,0,0)=
\begin{cases}
0 & \mbox{ if } 1\leq j\leq q-1,\\
-\partial_3 u(1,0,0) g^{(q)}(0) &\mbox{ if } j=q.
\end{cases}
\end{equation}
Hence, the Taylor expansion of $u(1,z_2,z_3)$ leads to 
\begin{equation*}
     \lim_{t\to\infty}\sqrt{t} u\left(1,\frac{v_2}{v_1t^{1/(2q)}}, \frac{v_3}{v_1\sqrt{t}}\right)=\partial_3 u(1,0,0) \left(\frac{v_3}{v_1}-\frac{g^{(q)}(0)}{q!}\frac{v_2^q}{v_1^q}\right).
\end{equation*}     
The proof that this convergence is dominated is deferred to Lemma \ref{lem:domination_case_F3} below, where the crucial role of $C(\theta_c,\eta)$ will appear clearly. Therefore, as $t\to\infty$,
\begin{multline}
\label{eq:J_delta_asymptotic}
     K_{\delta, \eta,\theta_c}(t)= t^{-1-{1}/{(2q)}}\partial_3 u(1,0,0) \\
     \times\int_{D}v_1^{p_1-q}(v_1^{q-1}v_3-cv_2^q)e^{-v_1^2/2}e^{-\vert a\vert v_3}\text{d}v+o(t^{-1-{1}/{(2q)}}).
\end{multline}
Notice that the last integral is positive since $D$ has positive (infinite) Lebesgue measure and is exactly the domain where the integrand is positive. We now compute its value. Since $q$ is even, for any fixed $v_1>0$ and $v_3>0$, we have
\begin{equation*}
     \int_{\{v_2\in\mathbb R: v_1^{q-1}v_3>cv_2^q\}}(v_1^{q-1}v_3-cv_2^q)\text{d}v_2=2\left(1-\frac{1}{q+1}\right)(c^{-1}v_1^{q-1}v_3)^{1+{1}/{q}}.
\end{equation*}     
Thus, by an application of Fubini's theorem, the integral in \eqref{eq:J_delta_asymptotic} becomes
\begin{equation*}
     2\left(1-\frac{1}{q+1}\right)c^{-1-{1}/{q}}\int_{0}^\infty v_1^{p_1-{1}/{q}}e^{-v_1^2/2}\text{d}v_1 \int_{0}^\infty v_3^{1+{1}/{q}}e^{-\vert a\vert v_3}\text{d}v_3,
\end{equation*}     
and can be expressed in terms of the Gamma function as
\begin{equation*}
   \frac{2^{(p_1+1-1/q)/{2}}(1-\frac{1}{q+1})}{\vert a\vert^{2+{1}/{q}} c^{1+{1}/{q}}}\Gamma\left(\frac{p_1+1-1/q}{2}\right)\Gamma\left(2+\frac{1}{q}\right)=
	\kappa(q)c^{-(1+{1}/{q})}.
\end{equation*}
This concludes the proof of Lemma \ref{lem:individual_contact_point_contribution}.
\end{proof}

\begin{lemma}
\label{lem:domination_case_F3}
Let $a_{i,j}$ denote the coefficient of  $z_2^iz_3^j$ in the Taylor expansion of $u(1,z_2,z_3)$ at $(0,0)$. If $\eta>0$ in the definition \eqref{eq:thin_cones_definition} of $C(\theta_c,\eta)$ is small enough, then
\begin{equation*}
     \int_{D_t}v_1^{p_1}\left\vert\sqrt{t} u\left(1,\frac{v_2}{v_1t^{1/(2q)}}, \frac{v_3}{v_1\sqrt{t}}\right)-\left(a_{0,1}\frac{v_3}{v_1}+a_{q,0}\frac{v_2^q}{v_1^q}\right)\right\vert e^{-v_1^2/2}e^{-\vert a\vert v_3}\textnormal{d}v= o(1),\quad t\to\infty.
\end{equation*}     
\end{lemma}

\begin{proof}
Since the function $u(1,z_2,z_3)$ can be extended to a function infinitely differentiable in a neighborhood of $(0,0)$, see \cite[Theorem A]{MoNi57}, there exists $M>0$ such that, for $\eta_0>0$ small enough,
\begin{equation*}
     u(1,z_2,z_3)=\sum_{i+j\leq q}a_{i,j}z_2^iz_3^j+\vert (z_2,z_3)\vert^{q+1}R(z_2,z_3),
\end{equation*}     
where $\vert R(z_2,z_3)\vert\leq M$ for all $(z_2,z_3)\in B(0,\eta_0)$. We already know (see \eqref{eq:link_between_partial_derivatives} in the proof of Theorem \ref{thm:case:F3}) that $a_{i,0}=0$ for all $i\in\{0,\ldots ,q-1\}$, hence
\begin{equation*}
     \vert u(1,z_2,z_3)-(a_{0,1}z_ 3+a_{q,0}z_2^q)\vert\leq \sum_{2\leq j\leq q}\vert a_{0,j} z_3^j\vert+\sum_{\substack{i,j\geq 1\\ i+j\leq q}}\vert a_{i,j}z_2^i z_3^j\vert+\vert (z_2,z_3)\vert^{q+1}M.
\end{equation*}     
Let $\epsilon\in(0,1)$ be fixed. For $(z_2,z_3)\in B(0,\eta_0)$, we use the upper bound
\begin{equation*}
     \vert a_{0,j}\vert\vert z_3\vert^{1+\epsilon}\eta_0^{j-(1+\epsilon)},\quad \forall j\geq 2,
\end{equation*}     
for the terms inside of the first sum, and the upper bound
\begin{equation*}
     \vert a_{i,j}\vert \vert z_2\vert\vert z_3\vert \eta_0^{i+j-2},\quad \forall i+j\geq 2,
\end{equation*}     
for the terms inside of the second sum. For the last term, we write
\begin{equation*}
     \vert (z_2,z_3)\vert^{q+1}\leq C (\vert z_2\vert^{q+1}+\vert z_3\vert^{q+1})\leq C(\vert z_2\vert^{q+1}+\vert z_3\vert^{1+\epsilon}\eta_0^{q-\epsilon}),
\end{equation*}     
and we finally obtain the upper bound
\begin{equation}
\label{eq:nice_candidate_for_domination}
\vert u(1,z_2,z_3)-(a_{0,1}z_ 3+a_{q,0}z_2^q)\vert\leq C_1\vert z_3\vert^{1+\epsilon}+C_2\vert z_2\vert \vert z_3\vert+ C_3\vert z_2\vert^{q+1},
\end{equation}
where $C_1,C_2,C_3>0$ are positive constants (depending on $\eta_0$ and $\epsilon$ only).

On the other hand, the definition of $C(\theta_c,\eta)$ ensures that 
\begin{equation*}
     \left\vert \left(\frac{v_2}{v_1t^{1/(2q)}},\frac{v_3}{v_1\sqrt{t}}\right)\right\vert\leq \eta+o(\eta),\quad \eta\to0,
\end{equation*}     
for all $(v_1,v_2,v_3)\in D_t$. Therefore, if $\eta>0$ is small enough so that $\eta+o(\eta)\leq \eta_0$, then according to 
\eqref{eq:nice_candidate_for_domination} we have
\begin{multline*}
     \left\vert \sqrt{t}u\left(1,\frac{v_2}{v_1t^{1/(2q)}},\frac{v_3}{v_1\sqrt{t}}\right)-\left(a_{0,1}\frac{v_3}{v_1}+a_{q,0}\frac{v_2^q}{v_1^q}\right)\right\vert
\\\leq o(1)\left( C_1\left\vert \frac{v_3}{v_1}\right\vert^{1+\epsilon}+C_2\left\vert \frac{v_2}{v_1}\right\vert \left\vert \frac{v_3}{v_1}\right\vert+ 
C_3\left\vert \frac{v_2}{v_1}\right\vert^{q+1} \right),
\end{multline*}
(where $o(1)$ is a function of $t$ alone)
for all $(v_1,v_2,v_3)\in D_t$, and the result follows from Lemma~\ref{lem:integrability_of_the_limit_on_D} below, provided that $\epsilon$ has been chosen so small
that $1+\epsilon+1/q\leq 2$.
\end{proof}

\begin{lemma}
\label{lem:integrability_of_the_limit_on_D}
The integral
\begin{equation*}
     \int_D v_1^{p_1}\left\vert\frac{v_2}{v_1}\right\vert^{\alpha}\left\vert\frac{v_3}{v_1}\right\vert^{\beta} e^{-v_1^2/2}e^{-\vert a\vert v_3}\textnormal{d}v
\end{equation*}     
is finite for all $\alpha, \beta\geq 0$ such that $\beta+(\alpha+1)/q\leq 2$.
\end{lemma}
\begin{proof}
Using Fubini's theorem, this integral can be shown to be equal to
\begin{equation*}
     \int_{0}^\infty v_1^{p_1+1-\beta-(\alpha+1)/q}e^{-v_1^2/2}\text{d}v_1 \int_{0}^\infty v_3^{\beta+(\alpha+1)/q}e^{-\vert a\vert v_3}\text{d}v_3,
\end{equation*}     
up to some positive multiplicative constant. The result follows since $p_1>0$.
\end{proof}

\begin{lemma}
\label{lem:implicit_vs_parametric_derivatives}
 Let $n\geq 1$ and $k\geq 0$, and assume that $f: \RR^{n+1}\to\RR$ and $g:\RR\to\RR$, with $g(0)=0$, are two functions infinitely differentiable such that for some constant $c$,
\begin{equation}
\label{eq:implicit_relation}
f(x,g(x), g'(x),\ldots,g^{(n-1)}(x))=c,
\end{equation} 
for all $x$ in some neighborhood of $x=0$, and
\begin{equation}
\label{eq:zero_derivatives}
g'(0)=g^{(2)}(0)=\cdots=g^{(n-1+k)}(0)=0.
\end{equation}
Then 
\begin{equation*}
\partial_{1,1,\ldots,1}^{(k+1)}f(0)=-\partial_{n+1}f(0)g^{(n+k)}(0).
\end{equation*}
\end{lemma}

\begin{proof}
Let $H(n,k)$ denote the statement that the conclusion of the lemma is true for the pair $(n,k)$. We shall prove that
\begin{itemize}
\item $H(n,0)$ holds for all $n\geq 1$;
\item For all $n\geq 1$ and $k\geq 1$, $H(n+1,k-1)$ implies $H(n,k)$.
\end{itemize}
The lemma will clearly follow by induction. 

Let $f$ and $g$ be two functions satisfying the hypotheses of Lemma~\ref{lem:implicit_vs_parametric_derivatives} for some $n\geq 1$ and $k\geq 0$, and set $\gamma(x)=(x,g(x),g'(x),\ldots,g^{(n-1)}(x))$. First, differentiating relation~\eqref{eq:implicit_relation} w.r.t.\ the variable $x$ shows that
\begin{equation}
\label{eq:implicit_relation_on_derivatives}
\partial_1f(\gamma(x))+\sum_{j=2}^n\partial_jf(\gamma(x))g^{j-1}(x)+\partial_{n+1}f(\gamma(x))g^{(n)}(x)=0,
\end{equation}
for all $x$ in some neighborhood of $0$. Hence, according to \eqref{eq:zero_derivatives}, we get that
\begin{equation*}
\partial_1f(0)+\partial_{n+1}f(0)g^{(n)}(0)=0,
\end{equation*}
thereby proving $H(n,0)$. Furthermore, equation~\eqref{eq:implicit_relation_on_derivatives} can be rewritten as
\begin{equation*}
h(x,g(x),g'(x),\ldots, g^{(n)}(x))=0,
\end{equation*}
in some neighborhood of $x=0$, where $h:\RR^{n+2}\to\RR$ is defined by
\begin{equation}
\label{eq:h_function_definition}
h(x_1,x_2,x_3,\ldots,x_{n+2})=\partial_1f(\gamma(x))+\sum_{j=2}^n\partial_jf(\gamma(x))x_{j+1}+\partial_{n+1}f(\gamma(x))x_{n+2}.
\end{equation}
Since equation \eqref{eq:zero_derivatives} is left invariant when replacing $n$ by $n+1$ and $k$ by $k-1$,  functions $h$ and $g$ fulfill the hypotheses of the lemma for the pair $(n+1,k-1)$. Therefore, if $H(n+1,k-1)$ holds, then
\begin{equation*}
\partial_{1,1,\ldots,1}^{(k)}h(0)=-\partial_{n+2}h(0)g^{(n+k)}(0).
\end{equation*}
But it is clear from the definition~\eqref{eq:h_function_definition} of $h$ that
\begin{equation*}
\partial_{1,1,\ldots,1}^{(k)}h(0)=\partial_{1,1,\ldots,1}^{(k+1)}f(0),
\end{equation*}
and
\begin{equation*}
\partial_{n+2}h(0)=\partial_{n+1}f(0).
\end{equation*}
Hence $H(n,k)$ holds, and the proof is completed.
\end{proof}

\section*{Acknowledgments}
This work was partially supported by Agence Nationale de la Recherche Grant ANR-09-BLAN-0084-01. We thank Alano Ancona, Guy Barles, Emmanuel Lesigne and Marc Peign\'e for interesting discussions. We thank a referee and an associate editor for their useful comments and suggestions.

\end{document}